\documentclass[9pt,a4paper,oneside]{article}
\usepackage{amsmath,amsfonts,amsthm}
\usepackage{mathtools}
\usepackage[dvipsnames]{xcolor}
\usepackage[disable]{todonotes}
\usepackage{ifthen}
\usepackage{xfp}
\usepackage{amssymb}
\usepackage{caption}
\usepackage{subcaption}
\usepackage{url}
\usepackage{authblk}
\usepackage{natbib}

\usepackage{tikz}
\usetikzlibrary{decorations.markings}
\usetikzlibrary{shapes,graphs,arrows}
\usetikzlibrary{calc,shapes.callouts,shapes.arrows}
\tikzset{->-/.style={decoration={markings,
  mark=at position .5 with {\arrow{>}}},postaction={decorate}}}
\tikzset{->--/.style={decoration={markings,
  mark=at position .25 with {\arrow{>}}},postaction={decorate}}}
\tikzset{-<-/.style={decoration={markings,
  mark=at position .45 with {\arrow{<}}},postaction={decorate}}}
\tikzset{->>-/.style={decoration={markings,
  mark=at position .6 with {\arrow{>}}},postaction={decorate}}}
\tikzset{-<>-/.style={decoration={markings,
  mark=at position .6 with {\arrow{>}},mark=at position .45 with {\arrow{<}}},postaction={decorate}}}
\usetikzlibrary{shapes}
\pgfdeclarelayer{bg}    
\pgfsetlayers{bg,main}
\usepackage{enumerate}
\usepackage{enumitem}
\usepackage{comment}
\usepackage{hyperref}

\title{Linear extensions and directed clique counts via modular partitions}
\author[1]{Daniela Egas Santander}
\author[2]{Matteo Santoro}
\author[3]{Jason P. Smith}
\affil[1]{Max Planck Institute of Molecular Cell Biology and Genetics and Center for Systems Biology, Dresden, Germany}
\affil[2]{SISSA, Trieste, Italy}
\affil[3]{Nottingham Trent University, Nottingham, UK}
\date{}

\newtheorem{theorem}{Theorem}
\counterwithin{theorem}{section}
\newtheorem{lemma}[theorem]{Lemma}
\newtheorem{lemma_definition}[theorem]{Lemma/Definition}
\newtheorem{proposition}[theorem]{Proposition}
\newtheorem{example}[theorem]{Example}
\newtheorem{remark}[theorem]{Remark}
\newtheorem{corollary}[theorem]{Corollary}

\newtheorem{definition}[theorem]{Definition}
\newtheorem{question}[theorem]{Question}
\newtheorem{notation}[theorem]{Notation}

\newtheorem*{claim*}{Claim}

\newcommand{\cM}{\mathcal{M}}

\newcommand{\cQ}{\mathcal{Q}}
\newcommand{\cP}{\mathcal{P}}
\newcommand{\cS}{\mathcal{S}}
\newcommand{\bn}{\mathbf{n}}

\newcommand{\CG}{\mathrm{CG}}

\newcommand{\IG}{\mathrm{IG}}

\newcommand{\lex}{\mathrm{LE}}

\begin{document}
	\maketitle
	

\begin{abstract}
    Counting linear extensions is a fundamental problem in poset theory. It is known to be \#P-complete, with polynomial-time formulas available in special cases. In this work, we develop new recursive formulas for counting linear extensions of posets whose modular partitions have particular structure. Specifically, we focus on posets whose incomparability graph has a modular partition with a skeleton that is a tree, a necklace of cliques, or a combination of both. The proofs are constructive and allow for the explicit generation of all linear extensions. We also discuss equivalent formulations of the problem in terms of permutations and directed graphs. 
    The directed graph perspective is related to counting directed simplices in the directed flag complex of a digraph, with applications to understanding higher-order structure in neural circuits.
\end{abstract}

\section{Introduction}
A linear extension of a poset is a total order of the poset elements, which is compatible with its partial order. 
Counting the number of linear extensions of a poset is one of the most fundamental problems in the study of posets, and it is known to be $\# P$-complete \cite{BW91}.
Polynomial time formulas exist for particular classes of graphs, including
 trees \cite{Atk90}, series-parallel posets \cite{Val79}, d-complete posets \cite{Pro14}, and mobile posets \cite{Gar21}, 
 see \cite[Section 12.1]{Cha23} for a comprehensive overview. However, all these classes form a vanishingly small proportion of all posets. 
 On the other hand, numerous bounds for the number of linear extensions exist, of varying tightness and complexity, see~\cite{Cha23}.
Many computationally difficult problems in poset theory can be made tractable by decomposing the poset into simpler parts.
In particular, polynomial time algorithms exist for posets with bounded \emph{treewidth} \cite{Eib19} and \emph{modular width} \cite{Hab87,Dien23}.

A modular partition of a graph is a decomposition of a graph into modules, where all vertices within each module have the same neighbours outwith the module. A modular partition of the comparability graph of a poset can be used to decompose the poset, 
where the vertices of the comparability graph are the elements of the poset and there is an edge between any pair of comparable elements. 
The modular width of a modular partition is the size of the largest module.
In \cite{Hab87} it is shown that given a modular partition of a poset $\cP$, the number of linear extensions of $\cP$ is given by computing the number of linear extensions of each module and the number of linear extensions of the poset obtained by replacing each module by a totally ordered set of the same size. 
This is used to show that the number of linear extensions can be computed in polynomial time, thus the number of linear extensions is computationally bounded by computing the linear extensions of each module.

We prove new recursive formulas for the number of linear extensions when the modular partition has certain structure.
If we further restrict to cases where the modules are always cliques or independent sets, we obtain closed formulas.
The posets considered are those with a modular partition of the comparability graph whose \emph{skeleton graph} is the complement of a tree or necklace of cliques (or some combination of these), where the skeleton of a modular partition is obtained by replacing every module with a single vertex.
The proof of our formula is constructive, thus can be used to explicitly list all linear extensions. See Figure~\ref{fig:main_results} for examples of the main results.

Counting linear extensions has equivalent formulations in other fields of mathematics.
They are closely linked to permutations, since a linear extension can be thought of as a permutation of the poset elements.
Counting linear extensions is equivalent to counting permutations where the inversion set is a subset of the set of incomparable pairs of the poset.
This equivalent reformulation has brought forth several applications: estimating the complexity of counting linear extensions \cite{BW91, dittmer2018counting}, methods to uniformly sample linear extensions \cite{PR94,huber2025generating} and to estimate the number of linear extensions \cite{banks2010using}. Moreover, in the other direction linear extensions have been used to understand permutation statistics \cite{BjWa91} and count permutation patterns and related concepts \cite{EN12,yakoubov2015pattern,cooper2016complexity}.
On the other hand, counting linear extensions is equivalent to counting the number of edge-induced acyclic tournaments in an acyclic tournament with added reciprocal edges. This perspective, in turn, models a problem in topology concerning how the number of directed simplices in an oriented directed graph changes when reciprocal connections are added.

These different interpretations of the counting problem allow this question to address several applications across the sciences.
In decision making, these methods can be used to estimate ranking probabilities from partial orders \cite{lerche2003evaluation} and in determining the costs of sorting algorithms \cite{schonhage1976production}.
They also arise in computer science, where counting linear extensions can be used to determine efficient schedules for parallel or constrained tasks \cite{cato1995parallel,crampton2004algebraic}.
Finally, the directed graph interpretation has applications in the sciences, particularly in neuroscience, where it is known that neural networks exhibit an over-expression of simplices and reciprocal connections, both of which affect network function~\cite{sizemore2018cliques, tadic2019functional, sizemore2019importance, andjelkovic2020topology, shi2021computing, ecker2024cortical, santander2025heterogeneous}. Counting linear extensions captures the interaction between reciprocal connections and directed simplices in such networks.

The structure of the paper is as follows.  In Section 2 we introduce the necessary background on linear extensions and modular partitions. In Section 3 we present three formulas obtained by restricting our main result to specific cases. In Section 4 we present our most general result, its proof and an exploration of when a modular partition of the comparability graph is not a modular partition of the associated poset. We finish by presenting some applications in Section 5.

\section{Preliminaries}
\subsection{Basic definitions, notation, and conventions}

We begin by establishing notation and recalling some standard graph theory terminology, for any graph theory definitions we omit see \cite{Die12} for a detailed background.
Unless otherwise stated we consider all graphs to be directed, with no loops or multi edges (i.e. at most one edge from $i$ to $j$) but we do allow bidirectional connections, that is, the edges from $i$ to~$j$ and from $j$ to $i$ can exist at the same time.

\begin{notation}
We denote by $[n]$ the set $[n]=\{1,2,\ldots, n\}$.
For a graph $G$, we denote the vertex and edge sets of $G$ by $V(G)$ and $E(G)$, respectively.
We denote by $i\to j$ an edge from $i$ to $j$ in~$G$. 
If $G$ is undirected we denote by $i- j$ an edge between $i$ and $j$ in $G$. 
The \emph{neighbours} of a vertex $v$ are the vertices in $V(G)$ that are connected by an edge to $v$.  
If $G$ is directed, these vertices can be further classified into \emph{in-neighbours} and \emph{out-neighbours}, depending if they are connected via an edge to $v$ or from $v$, respectively.
\end{notation}

\begin{definition}[Subgraphs]
 An \emph{edge-induced} subgraph $H$ of $G$ is obtained by taking a set $X\subseteq E(G)$ and setting $E(H)=X$ and $V(H)$ as all end points in $X$. 
 A \emph{vertex-induced} subgraph $H$ of $G$ is obtained by taking a set $X\subseteq V(G)$ and setting $V(H)=X$ and $E(H)$ edges in $E(G)$ where both end points are in $X$, and is denoted by $G|_X$. 
\end{definition}

\begin{definition}[Transitive graphs]
A graph $G$ is said to be \emph{transitive} if whenever $i\to j$ and $j \to k$ are edges in $G$ then $i\to k$ is an edge in $G$. 
The \emph{transitive closure} of a graph $G$ is the smallest transitive graph containing $G$, where by smallest we mean it has the minimal number of edges.  
Equivalently, the transitive closure of $G$ is the graph with the same vertex set as $G$ and edge $i \to j$ whenever there is a directed path from $i$ to $j$ in $G$.
A \emph{transitive ordering} of an undirected graph $G$ is a total ordering~$\prec$ of~$V(G)$ such that if every edge is oriented according to $\prec$, then the resulting directed graph is transitive.
\end{definition}

\begin{definition}
The \emph{complement} $\bar{G}$ of a graph $G$ on $n$ vertices is obtained by $V(\bar{G})=V(G)$ and~$E(\bar{G})=K\setminus E(G)$, where $K$ is the set of all $2$-element subsets of $V(G)$ if $G$ is undirected, and $K$ is the set of all $2$-element ordered pairs of $V(G)$ is $G$ is directed.
\end{definition}

Next we recall the definition of a poset, and associated terminology, for any poset definitions we omit see \cite[Chapter 3]{Sta11} for a detailed background.

\begin{definition}
A \emph{partially ordered set}, or \emph{poset}, is a pair $(\cP,\leq_\cP)$ where $\cP$ is a set and $\leq_\cP$ is a binary relation which is reflexive, antisymmetric and transitive.  If there is no possibility of confusion we will denote $\leq_\cP$ simply by $\leq$ and the tuple  $(\cP,\leq_\cP)$ simply by $\cP$.
Let $s,t \in \cP$,  we write $s<t$ if~$s\leq t$ and $s\neq t$.
We say $t$ \emph{covers} $s$ if~$s< t$ and $[s,t]:=\{u\in P\,|\, s\leq u \leq t\} =\{s,t\}$.  
Given a subset~$S\subseteq\mathcal{P}$, we denote by $\mathcal{P}_{|S}$ the poset on $S$ with the same order relationships as $\mathcal{P}$. 
We say $s$ and~$t$ are \emph{comparable} if either $s\leq t$ or $t\leq s$ and \emph{incomparable} otherwise.

We say that a poset is \emph{totally ordered} or \emph{has a total order} if every pair of elements is comparable, and in this case we call it a \emph{chain}.
Conversely, an \emph{anti-chain} is a poset where every pair of elements is incomparable. 
\end{definition}

\begin{example}
We denote by $\bn$ the chain poset with underlying set $[n]$ with its usual order.  Note that in $\bn$, for any $1\leq i < n$, $i+1$ is the only element that covers $i$.
\end{example}

We can construct several graphs that represent the structure of a poset, see Figure \ref{fig:graphs_poset} for examples of each of the following such graphs.  

\begin{definition}[Graphs associated to a poset]
Any poset $\cP$ has an associated transitive acyclic graph, which we denote by $G_\cP$, 
whose vertices are the elements of $\cP$ and  edges are $s \to t$ whenever~$s< t$.  
In fact, this construction defines a bijection between poset structures on the set $\cP$ and transitive acyclic graphs with vertex set $\cP$.  

The \emph{comparability graph} of $\cP$, denoted by $\CG_\cP$, is the undirected graph with vertices the elements of $\cP$ and an edge $s-t$ whenever $s$ and $t$ are comparable.  In particular, the comparability graph of $\cP$, can be obtained from $G_\cP$ by forgetting the orientation of the edges.
 
The \emph{incomparability graph} of $\cP$, denoted by $\IG_\cP$ is the undirected graph with vertices $\cP$, and an edge $s-t$ whenever $s$ and $t$ are incomparable. 
Thus, the comparability graph is the complement of the incomparability graph.

The \emph{Hasse diagram} of $\cP$ is the undirected graph with vertex set $\cP$ and with edges corresponding to covering relations.  We draw the Hasse diagram of $\cP$ such that if $s\leq t$ then $t$ is drawn ``above" $s$.

\end{definition}
\begin{figure}[h!]
\begin{center}\begin{tikzpicture}[scale=1]
  \def\s{.75}
  \def\h{2}
  \node (a) at (-1.5,\h){\textbf{A}};
  \node[circle, draw=black, scale=\s] (1) at (0,0.75){$1$};
  \node[circle, draw=black, scale=\s] (2) at (1,0){$2$};
  \node[circle, draw=black, scale=\s] (3) at (0.5,-1){$3$};
  \node[circle, draw=black, scale=\s] (4) at (-0.5,-1){$4$};
  \node[circle, draw=black, scale=\s] (5) at (-1,0){$5$};
  \draw[thick,->-] (2) to (1);
  \draw[thick,->-] (3) to (1);
  \draw[thick,->-] (4) to (1);
  \draw[thick,->-] (5) to (1);
  \draw[thick,->-] (3) to (2);
  \draw[thick,->-] (4) to (2);
  \draw[thick,->-] (5) to (2);
  \draw[thick,->-] (4) to (3);
  \draw[thick,->-] (5) to (3);
   \draw[thick,->-] (5) to (4);

  \def\x{4}
  \node (b) at (-1.5+\x,\h){\textbf{B}};
  \node[circle, draw=black, scale=\s] (1) at (0+\x,0.75){$1$};
  \node[circle, draw=black, scale=\s] (2) at (1+\x,0){$2$};
  \node[circle, draw=black, scale=\s] (3) at (0+\x.5,-1){$3$};
  \node[circle, draw=black, scale=\s] (4) at (-0.5+\x,-1){$4$};
  \node[circle, draw=black, scale=\s] (5) at (-1+\x,0){$5$};
  \draw[thick,-] (1) to (2);
  \draw[thick,-] (1) to (3);
  \draw[thick,-] (1) to (4);
  \draw[thick,-] (1) to (5);
  \draw[thick,-] (2) to (3);
  \draw[thick,-] (2) to (4);
  \draw[thick,-] (2) to (5);
  \draw[thick,-] (3) to (4);
  \draw[thick,-] (3) to (5);
   \draw[thick,-] (4) to (5);
  
  \def\x{8}
  \node (c) at (-1.5+\x,\h){\textbf{C}};
  \node[circle, draw=black, scale=\s] (1) at (0+\x,0.75){$1$};
  \node[circle, draw=black, scale=\s] (2) at (1+\x,0){$2$};
  \node[circle, draw=black, scale=\s] (3) at (0+\x.5,-1){$3$};
  \node[circle, draw=black, scale=\s] (4) at (-0.5+\x,-1){$4$};
  \node[circle, draw=black, scale=\s] (5) at (-1+\x,0){$5$};
  
  \def\x{12}
  \node (d) at (-1.5+\x,\h){\textbf{D}};
  \node[circle, draw=black, scale=\s] (5) at (\x,1.5){$5$};
  \node[circle, draw=black, scale=\s] (4) at (\x,1.5-3/4){$4$};
  \node[circle, draw=black, scale=\s] (3) at (\x,1.5-3*2/4){$3$};
  \node[circle, draw=black, scale=\s] (2) at (\x,1.5-3*3/4){$2$};
  \node[circle, draw=black, scale=\s] (1) at (\x,-1.5){$1$};
   \draw[thick,-] (1) to (2);
   \draw[thick,-] (2) to (3);
   \draw[thick,-] (3) to (4);
   \draw[thick,-] (4) to (5);
 \end{tikzpicture}\end{center}
 \caption{Graphs associated to the chain poset $\mathbf{4}$. 
 \textbf{A:} The graph $G_\mathbf{4}$. 
 \textbf{B:} The comparability graph~$\CG_\mathbf{4}$. 
 \textbf{C:} The incomparability graph $\IG_\mathbf{4}$. 
 \textbf{D:} The Hasse diagram of $\mathbf{4}$.}
 \label{fig:graphs_poset}
  \end{figure}
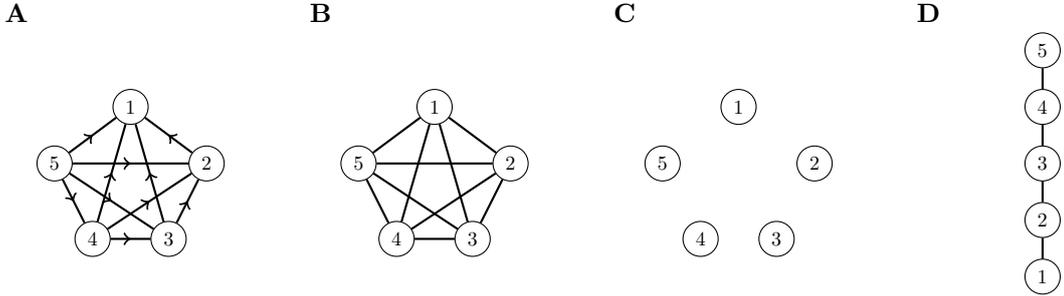

\subsection{Linear extensions of a poset and modular partitions}
We recall the definition of a linear extension of a poset.
\begin{definition}
Given a poset $(\cP,\le)$ a \emph{linear extension} of $\cP$ is a total ordering $\prec$ of the elements of~$\cP$ which preserves the partial order of $\cP$, that is, if $x\le y$ then $x\prec y$.
Equivalently, a linear extension is an order preserving bijection $\ell:\cP\rightarrow \bn$, where $n=|\cP|$ and $\bn$ is the chain with $n$ elements.
We denote the number of linear extensions of a poset $\cP$ by $\lex(\cP)$.
\end{definition}

Counting the number of linear extensions of a poset is a classical problem in poset theory. It is known to be \#P-complete \cite{BW91}, and many interesting sequences of numbers can be found when counting linear extensions. For example, the number of linear extensions of the poset $\mathbf{2}\times\bn$ is the Catalan number $C_n$ \cite[Exercise 177]{stanley2015catalan}, where $\mathbf{2}\times\bn$ is the Cartesian product of the $2$-element chain and $n$-element chain. Also, in \cite{Bri99} they construct the poset $L_n$ whose number of linear extensions is the Fibonacci numbers $F_n$. 

Our main tool for computing the number of linear extensions is modular partitions.
\begin{definition}[Modular partition]
Let $G$ be an undirected graph.  A \emph{module} of $G$ is a non-empty set~$M\subseteq V(G)$ such that all vertices of $M$ have the same neighbours in $V(G)\setminus M$.
 A \emph{modular partition} of $G$ is a partition of its vertex set into modules.
\end{definition}

The definition of a modular partition can be adapted for directed graphs, and thus also for posets. 

\begin{definition}[Modular partition of a digraph and poset]
Let $G$ be a directed graph. A module in~$G$ is a subset $M\subseteq V(G)$ such that all vertices in $M$ have the same in-neighbours and out-neighbours in $V(G)\setminus M$.
 A modular partition of $G$ is a partition of the vertices into modules.
A modular partition $\cM=\{ M_i\subseteq V(G)\}_{i=1}^k$ of a poset $\cP$ is a modular partition of its associated directed acyclic graph $G_\cP$, 
by abuse of notation where there is no confusion we use $M_i$ to denote $\cP|_{M_i}$.
\end{definition}

Note that a modular partition of a poset is a modular partition of the incomparability graph, but a modular partition of the incomparability does not necessarily give a modular partition of the poset.
 This is because whilst the modular partition of the incomparability graph will result in all modules having the same neighbours, it does not force the 
 orientation of the edges between neighbours to be consistent in the directed graph representation of the poset. 

For example, consider the transitive graph on three vertices $i,j,k$ with edges $i \to j$, $j\to k$, and~$i\to j$.  
The partition $M_1 =\{i,k\}$ and $M_2 =\{j\}$ is not a modular partition of the directed graph or its associated poset, but it is a modular partition of its (in)comparability graph.
We will explore this in more detail in Section~\ref{sec:notposet}.

Given a graph with a modular partition we can quotient all the modules together to give the following definition, see Figure \ref{fig:skeleton} for a simple example.

\begin{definition}[Skeleton of a modular partition]
Let $\cM=\{ M_i\subseteq V(G)\}_{i=1}^k$ be a modular partition of a graph $G$. The \emph{skeleton} of $\cM$ is the quotient graph $G/\cM$, where each module is identified to a single vertex.  More precisely, $G/\cM$ has vertices $\{v_1, v_2, \ldots, v_k\}$ and an edge $v_i \to v_j$ if there are edges in $G$ from the vertices in $M_i$ to the vertices in $M_j$.
\end{definition}

\begin{figure}[h!]
\begin{subfigure}{0.3\textwidth}
  \begin{center}
  \begin{tikzpicture}[every node/.style={circle, fill=black, radius=3pt, inner sep=0pt, minimum size=4pt}]
	\node[fill=white] (0) at (-1,3) {\textbf{A}};
	\node[fill=WildStrawberry] (1) at (0,0) {};
	\node[fill=WildStrawberry] (2) at (0.25,0.5) {};
	\node[fill=WildStrawberry] (3) at (-0.25, 0.5) {};
	\node[fill=WildStrawberry] (4) at (0, 1) {};
	\draw (1) -- (2) -- (4);
	\draw (1) -- (3) -- (4);
	\node[fill=blue] (5) at (0, 1.5) {};
	\node[fill=blue] (6) at (0, 2) {};
	\draw (4) -- (5) -- (6);
	\node[fill=JungleGreen] (7) at (0, 2.5) {};
	\node[fill=JungleGreen] (8) at (-0.5, 2.5) {};
	\node[fill=JungleGreen] (9) at (0.5, 2.5) {};
	\node[fill=JungleGreen] (10) at (0, 3) {};
	\draw (6) -- (7) -- (10);
	\draw (6) -- (8) -- (10);
	\draw (6) -- (9) -- (10);
	\node[fill=Goldenrod] (11) at (1.5, 1) {};
	\node[fill=Goldenrod] (12) at (1.75, 1.5) {};
	\node[fill=Goldenrod] (13) at (1.25, 1.5) {};
	\node[fill=Goldenrod] (14) at (1.5, 2) {};
	\draw (11) -- (12) -- (14);
	\draw (11) -- (13) -- (14);
  \end{tikzpicture}
  \end{center}
\end{subfigure}
\hfill
\begin{subfigure}{0.3\textwidth}
  \begin{center}
  \begin{tikzpicture}[every node/.style={circle, fill=black, radius=3pt, inner sep=0pt, minimum size=4pt}]
  		\node[fill=white] (0) at (-2,3) {\textbf{B}};
		\node[fill=WildStrawberry] (1) at (0,0) {};
		\node[fill=WildStrawberry] (2) at (0.25,0.5) {};
		\node[fill=WildStrawberry] (3) at (-0.25, 0.5) {};
		\node[fill=WildStrawberry] (4) at (0, 1) {};
		\node[fill=blue] (5) at (0, 1.5) {};
		\node[fill=blue] (6) at (0, 2) {};
		\node[fill=JungleGreen] (7) at (0, 2.5) {};
		\node[fill=JungleGreen] (8) at (-0.5, 2.5) {};
		\node[fill=JungleGreen] (9) at (0.5, 2.5) {};
		\node[fill=JungleGreen] (10) at (0, 3) {};
		\begin{scope}[yshift=3pt]
		\node[fill=Goldenrod] (11) at (1.5, 1) {};
		\node[fill=Goldenrod] (12) at (1.75, 1.5) {};
		\node[fill=Goldenrod] (13) at (1.25, 1.5) {};
		\node[fill=Goldenrod] (14) at (1.5, 2) {};	
		\end{scope}		
		\draw[dotted] (0, 2.6) ellipse (0.7cm and 0.5cm);
	    \draw[dotted] (1.5, 1.6) ellipse (0.5cm and 0.7cm);
	    \draw[dotted] (0, 1.75) ellipse (0.2cm and 0.5cm);
	    \draw[dotted] (0, 0.5) ellipse (0.5cm and 0.7cm);
	    \begin{pgfonlayer}{bg}
	    \draw[color=gray!30] (2) -- (3);
	    \draw[color=gray!30] (7) -- (8) -- (9);
	    \draw[color=gray!30] (12) -- (13);
		\foreach \n in {11, ..., 14}{
		\draw[color=gray!30] (1) -- (\n);
		\draw[color=gray!30] (2) -- (\n);
		\draw[color=gray!30] (3) -- (\n);
		\draw[color=gray!30] (4) -- (\n);
		\draw[color=gray!30] (5) -- (\n);
		\draw[color=gray!30] (6) -- (\n);
		\draw[color=gray!30] (7) -- (\n);
		\draw[color=gray!30] (8) -- (\n);
		\draw[color=gray!30] (9) -- (\n);
		\draw[color=gray!30] (10) -- (\n);	
	    }
	    \end{pgfonlayer}
		\node [draw=none, fill=none] (S1) at (2.25, 1.6) {$M_1$};
		\node [draw=none, fill=none] (S1) at (-1, 0.5) {$M_2$};
		\node [draw=none, fill=none] (S1) at (-1, 1.75) {$M_3$};
		\node [draw=none, fill=none] (S1) at (-1, 2.75) {$M_4$};
  \end{tikzpicture}
  \end{center}
\end{subfigure}
\hfill
\begin{subfigure}{0.3\textwidth}
	\begin{center}
	\begin{tikzpicture}[every node/.style={circle, fill=black, radius=3pt, inner sep=0pt, minimum size=4pt}]
		\node[fill=white] (0) at (-1,3) {\textbf{C}};
		\node[fill=WildStrawberry] (1) at (0,0) {};
		\node[fill=blue] (2) at (0, 1) {};
		\node[fill=JungleGreen] (3) at (0,2) {};
		\node[fill=Goldenrod] (4) at (1.5, 1.5) {};
		\draw (1) -- (4);
		\draw (2) -- (4);
		\draw (3) -- (4);
		\node [draw=none, fill=none] (S1) at (1.9, 1.5) {$M_1$};
		\node [draw=none, fill=none] (S1) at (-0.4, 0) {$M_2$};
		\node [draw=none, fill=none] (S1) at (-0.4, 1) {$M_3$};
		\node [draw=none, fill=none] (S1) at (-0.4, 2) {$M_4$};
	\end{tikzpicture}
	\end{center}
\end{subfigure}
\begin{subfigure}{\textwidth}
$$\textbf{D}\quad\quad \lex(\mathcal{P}) =  \binom{|\mathcal{P}|}{|M_1|}\prod_{s=1}^{N} \lex(M_s) $$
\end{subfigure}

\caption{
\textbf{Modular partition and its skeleton.}
\textbf{A:} The Hasse diagram of a poset $\cP$. 
\textbf{B:} The incomparability graph $\IG_\cP$ with a modular partition $\cM$.
\textbf{C:} The skeleton graph of $\IG_\cP/\cM$.
\textbf{D:} The formula for the number of linear extensions of $\cP$.}
\label{fig:skeleton}
\end{figure}
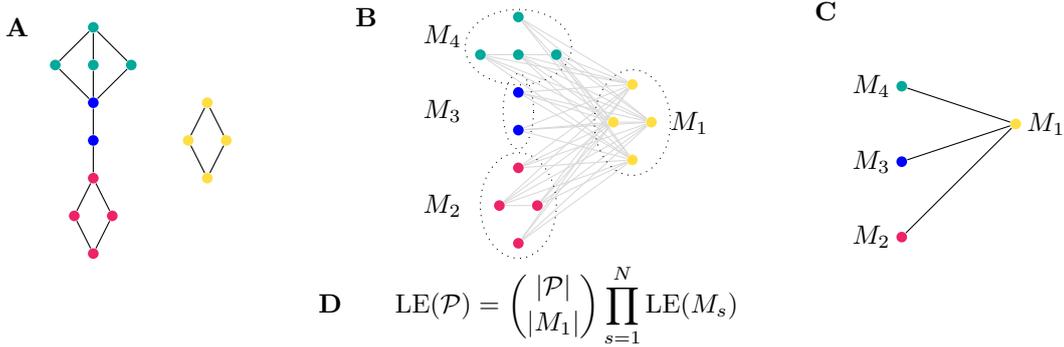

Taking the skeleton graph of a poset can be considered as the reverse of the lexicographic sum, defined as:

\begin{definition}
Let $\cS$ be a poset and $\Upsilon=\{\cP_s \}_{s\in \cS}$ a set of posets indexed by the elements of $\cS$.  The \emph{lexicographic sum of $\Upsilon$ over $\cS$}, which we denote by $\oplus_\cS \Upsilon$ is the poset $(\cP,\leq_\cP)$, where the set $\cP$ is the disjoint union of the sets in $\Upsilon$, i.e. $\cP=\sqcup_{s\in\cS}\cP_s$, and for $x\in\cP_s$, $y\in\cP_t$ we have that $x\leq y$ if either: 
$s<t$ in $\cS$ or
$s=t$ in $\cS$ and $x\leq y$ in $\cP_s$.

Two special cases of this construction are:
\begin{enumerate}
\item when $\cS$ is an antichain, this is the \emph{disjoint sum} and we denote it by $\cP_1+\cP_2+ \ldots +\cP_{|\cS|}$ and 
\item when $\cS$ is a chain, this is the \emph{ordinal sum} and we denote it by $\cP_1\oplus\cP_2\oplus\ldots \oplus\cP_{|\cS|}$.
\end{enumerate}
\end{definition}

\begin{remark}
The disjoint and ordinal sum are also sometimes called \emph{series} and \emph{parallel} operations, in particular in the setting of counting linear extensions.
\end{remark}

\begin{remark}
Note that if $\cM=\{ M_i\}_{i=1}^k$ is a modular partition of a poset $\cP$, then $\cP$ can be constructed as a lexicographic sum 
\[\cP=\oplus_{\cP/\cM} \Upsilon, \quad\quad \text{ where } \quad\quad \Upsilon=\{\cP|_{M_i} \}_{v_i\in \cP/\cM}.\]
\end{remark}

Any poset which can be built using only disjoint and ordinal sums is known as a \emph{series-parallel poset}, which are equivalent to $N$-free posets.
 Using the below formulas, the number of linear extensions of series-parallel posets can be easily computed.

\begin{lemma}\cite[Example 3.5.4]{Sta11}\label{lem:sums}
Given any two posets $\cP$ and $\cQ$ we have
$$\lex(\cP+\cQ)=\lex(\cP)\cdot\lex(\cQ)\,\,\,\,\,\,\,\,\,\text{and}\,\,\,\,\,\,\,\,\,\lex(\cP\oplus\cQ)=\binom{|\cP|+|\cQ|}{|\cP|}\cdot\lex(\cP)\cdot\lex(\cQ).$$
\end{lemma}

\section{Counting linear extensions in three special cases: necklaces, trees and paths}\label{sec:necktreepath}
Our main result provides a closed formula for the number of linear extensions of posets which have a modular partition whose skeleton graph has certain structure.
First we consider a simple result, which demonstrates our approach.
\begin{example}\label{ex:star}
	Let $\mathcal{P}$ be a poset and $\cM=\{M_1,\ldots,M_n\}$ a modular partition of the incomparability graph $\IG_\cP$ whose skeleton is a star graph, see Figure~\ref{fig:skeleton},
	then
	$$\lex(\mathcal{P}) =  \binom{|\mathcal{P}|}{|M_1|}\prod_{s=1}^{n} \lex(M_s). $$
    Here the result follows from Lemma~\ref{lem:sums} and by noting that $\mathcal{P}=M_1+\left(\bigoplus_{i=2}^n M_i\right)$.
    Note that this is an example of a series-parallel poset.
\end{example}

In this section we present three key corollaries of our main result Theorem~\ref{thm:joined}, similar to the approach above in Example~\ref{ex:star}. Each case corresponds to the structure of the skeleton graph of a modular partition of the incomparability graph given by a path, a necklace of $3$-cliques, and a full binary tree. See Figure~\ref{fig:main_results} for an example of each.

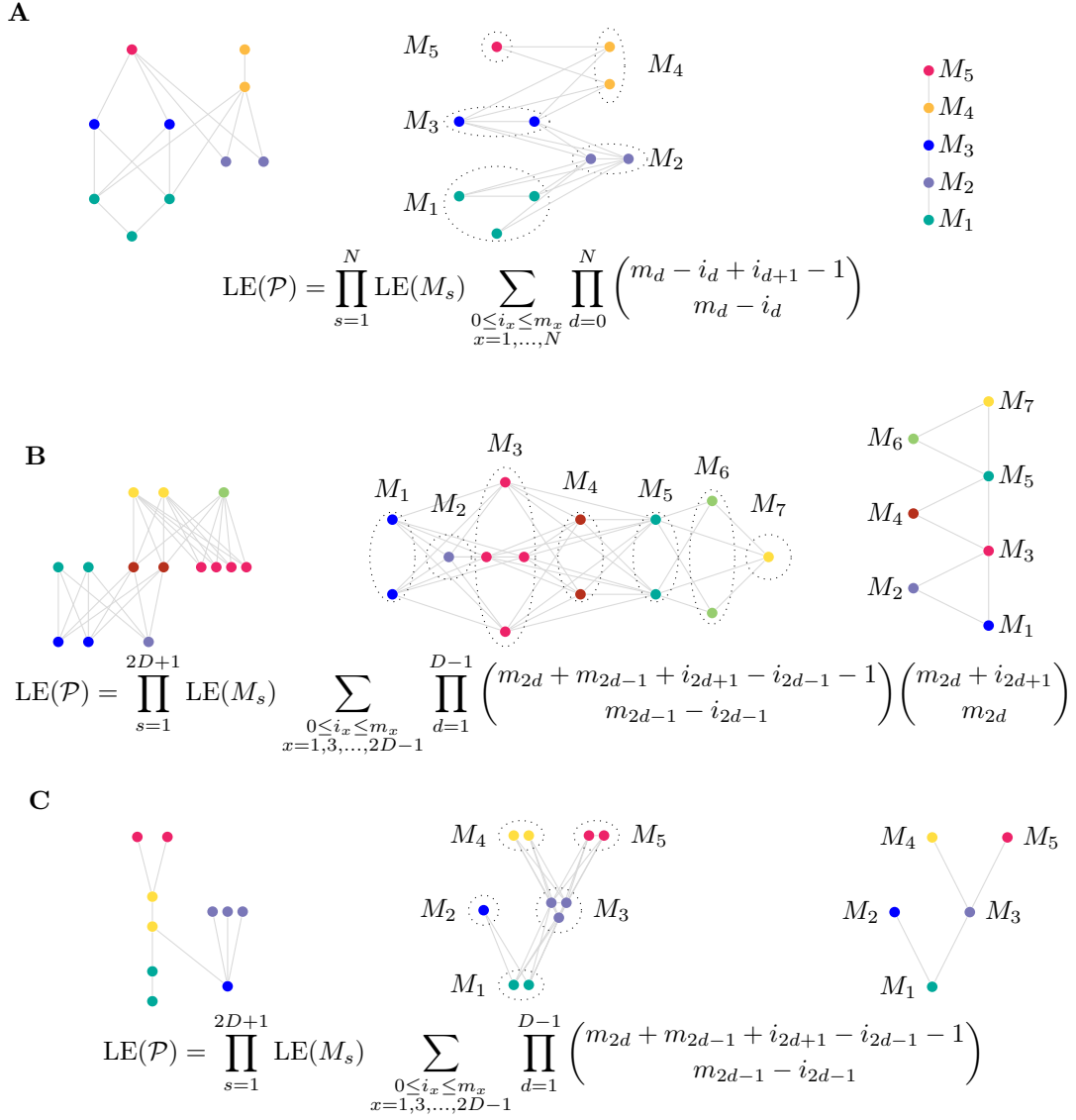
\begin{figure}
\begin{subfigure}{0.3\textwidth}
  \begin{center}
  \begin{tikzpicture}[every node/.style={circle, fill=black, radius=3pt, inner sep=0pt, minimum size=4pt}]
	\node[fill=white] (0) at (-1.5,0) {\textbf{A}};
	
	\node[fill=WildStrawberry] (1) at (0,-0.5) {};
	\node[fill=blue] (2) at (-0.5, -1.5) {};
	\node[fill=blue] (3) at (0.5, -1.5) {};
	\node[fill=JungleGreen] (4) at (-0.5, -2.5) {};
	\node[fill=JungleGreen] (5) at (0.5, -2.5) {};
	\node[fill=JungleGreen] (6) at (0, -3) {};
	\node[fill=Dandelion] (7) at (1.5, -0.5) {};
	\node[fill=Dandelion] (8) at (1.5, -1) {};
	\node[fill=Periwinkle] (9) at (1.75, -2) {};
	\node[fill=Periwinkle] (10) at (1.25, -2) {};
	\draw[color=gray!30] (4) -- (6) -- (5);
	\draw[color=gray!30] (4) -- (2) -- (5) -- (3) -- (4);
	\draw[color=gray!30] (2) -- (1) -- (3);
	\draw[color=gray!30] (9) -- (8) -- (10);
	\draw[color=gray!30] (7) -- (8);
	\draw[color=gray!30] (4) -- (8) -- (5);
	\draw[color=gray!30] (9) -- (1) -- (10);
  \end{tikzpicture}
  \end{center}
\end{subfigure}
\hfill
\begin{subfigure}{0.3\textwidth}
  \begin{center}
  \begin{tikzpicture}[every node/.style={circle, fill=black, radius=3pt, inner sep=0pt, minimum size=4pt}]

	\node[fill=WildStrawberry] (1) at (0,-0.5) {};
	\node[fill=blue] (2) at (-0.5, -1.5) {};
	\node[fill=blue] (3) at (0.5, -1.5) {};
	\node[fill=JungleGreen] (4) at (-0.5, -2.5) {};
	\node[fill=JungleGreen] (5) at (0.5, -2.5) {};
	\node[fill=JungleGreen] (6) at (0, -3) {};
	\node[fill=Dandelion] (7) at (1.5, -0.5) {};
	\node[fill=Dandelion] (8) at (1.5, -1) {};
	\node[fill=Periwinkle] (9) at (1.75, -2) {};
	\node[fill=Periwinkle] (10) at (1.25, -2) {};
	
	\draw[dotted] (0, -2.6) ellipse (0.7cm and 0.5cm);
	\draw[dotted] (1.5, -2) ellipse (0.5cm and 0.2cm);
	\draw[dotted] (0, -1.5) ellipse (0.7cm and 0.2cm);
	\draw[dotted] (1.5, -0.75) ellipse (0.2cm and 0.5cm);
	\draw[dotted] (0, -0.5) ellipse (0.2cm and 0.2cm);
	
	\node [draw=none, fill=none] (S2) at (-1, -2.6) {$M_1$};
	\node [draw=none, fill=none] (S2) at (2.25, -2) {$M_2$};
	\node [draw=none, fill=none] (S2) at (-1, -1.5) {$M_3$};
	\node [draw=none, fill=none] (S2) at (2.25, -0.75) {$M_4$};
	\node [draw=none, fill=none] (S2) at (-1, -0.5) {$M_5$};
	
	\draw[color=gray!30] (4) -- (5);	
	\draw[color=gray!30] (2) -- (3);
	\draw[color=gray!30] (9) -- (10);	
    \foreach \i in {4,5,6}{
        \foreach \j in {9,10}{
            \draw[color=gray!30] (\i) -- (\j);	
        }			
    }	
    \foreach \i in {2,3}{
        \foreach \j in {7,8,9,10}{
            \draw[color=gray!30] (\i) -- (\j);	
        }			
    }
	\draw[color=gray!30] (1) -- (7);	
	\draw[color=gray!30] (1) -- (8);
  \end{tikzpicture}
  \end{center}
\end{subfigure}
\hfill
\begin{subfigure}{0.3\textwidth}
	\begin{center}
	\begin{tikzpicture}[every node/.style={circle, fill=black, radius=3pt, inner sep=0pt, minimum size=4pt}]
    \node[circle, fill=JungleGreen, radius=3pt, inner sep=0pt, minimum size=4pt,label=right:$M_1$] (1) at (0,-2) {};
    \node[circle, fill=Periwinkle, radius=3pt, inner sep=0pt, minimum size=4pt,label=right:$M_2$] (2) at (0,-1.5) {};
    \node[circle, fill=blue, radius=3pt, inner sep=0pt, minimum size=4pt,label=right:$M_3$] (3) at (0,-1) {};
    \node[circle, fill=Dandelion, radius=3pt, inner sep=0pt, minimum size=4pt,label=right:$M_4$] (4) at (0,-.5) {};
    \node[circle, fill=WildStrawberry, radius=3pt, inner sep=0pt, minimum size=4pt,label=right:$M_5$] (5) at (0,0) {}; 
    \draw[color=gray!30] (1) -- (2) -- (3) -- (4) -- (5);
	\end{tikzpicture}
	\end{center}
\end{subfigure}
\begin{subfigure}{\textwidth}
  $$\lex(\mathcal{P})= \prod_{s=1}^N \lex(M_s) \sum_{\substack{0\le i_x\le m_{x}\\ x=1,\ldots,N}} \prod_{d=0}^{N} \binom{m_d - i_d + i_{d+1} - 1}{m_d - i_d}$$
\end{subfigure}\\[1em]
\begin{subfigure}{0.3\textwidth}
  \begin{center}
  \begin{tikzpicture}[every node/.style={circle, fill=black, radius=3pt, inner sep=0pt, minimum size=4pt}]
				\node[fill=white] (0) at (-1.5,4.5) {\textbf{B}};
				
				\node[fill = blue] (1) at (-1.2, 2) {};
				\node[fill = blue] (2) at (-0.8, 2) {};
				
				\node[fill = Periwinkle] (3) at (0,2) {};
	
				\node[fill = WildStrawberry] (4) at (1.3, 3) {};
				\node[fill = WildStrawberry] (5) at (1.1, 3) {};
				\node[fill = WildStrawberry] (6) at (0.9, 3) {};
				\node[fill = WildStrawberry] (7) at (0.7, 3) {};
				
				\node[fill = BrickRed] (8) at (-0.2, 3) {};
				\node[fill = BrickRed] (9) at (0.2, 3) {};

				\node[fill = JungleGreen] (10) at (-1.2, 3) {};
				\node[fill = JungleGreen] (11) at (-0.8, 3) {};
				
				\node[fill = YellowGreen] (12) at (1, 4) {};
				
				\node[fill = Goldenrod] (13) at (-0.2, 4) {};
				\node[fill = Goldenrod] (14) at (0.2, 4) {};
				
				\foreach \i in {1,2,3}{
					\foreach \j in {8,9,10,11}{
						\draw[color=gray!30] (\i) -- (\j);	
					}			
				}
				
				\foreach \i in {4,5,6,7,8,9}{
					\foreach \j in {12,13,14}{
						\draw[color=gray!30] (\i) -- (\j);
					}
				}
				
  \end{tikzpicture} 
  \end{center}
\end{subfigure}
\hfill
\begin{subfigure}{0.3\textwidth}
\begin{center}
  \begin{tikzpicture}[every node/.style={inner sep=0pt, minimum size=4pt, circle, fill=black}]
				\node[fill = blue] (1) at (-.5, -0.5) {};
				\node[fill = blue] (2) at (-.5, 0.5) {};
				\draw[dotted] (-.5, 0) ellipse (0.3cm and 0.6cm);
				\node [draw=none, fill=none] (I2) at (-.5, 0.9) {$M_1$};
								
				\node[fill = Periwinkle] (3) at (0.25, 0) {};
				\draw[dotted] (0.25, 0) ellipse (0.3cm and 0.3cm);
				\node [draw=none, fill=none] (S1) at (0.25, .7) {$M_2$};
				
				\node[fill = WildStrawberry] (4) at (.75, 0) {};
				\node[fill = WildStrawberry] (5) at (1,1) {};
				\node[fill = WildStrawberry] (6) at (1, -1) {};
				\node[fill = WildStrawberry] (7) at (1.25, 0) {};
				\draw[dotted] (1, 0) ellipse (0.4cm and 1.2cm);			
				\node [draw=none, fill=none] (S2) at (1, 1.5) {$M_3$};
				
				\node[fill = BrickRed] (8) at (2,-0.5) {};
				\node[fill = BrickRed] (9) at (2,0.5) {};
				\draw[dotted] (2, 0) ellipse (0.3cm and 0.6cm);	
				\node [draw=none, fill=none] (I1) at (2, 1) {$M_4$};

				\node[fill = JungleGreen] (10) at (3, 0.5) {};
				\node[fill = JungleGreen] (11) at (3, -0.5) {};
				\draw[dotted] (3, 0) ellipse (0.3cm and 0.6cm);
				\node [draw=none, fill=none] (I2) at (3, 0.9) {$M_5$};
				
				\node[fill = YellowGreen] (12) at (3.75, 0.75) {};
				\node[fill = YellowGreen] (13) at (3.75, -0.75) {};
				\draw[dotted] (3.75, 0) ellipse (0.3cm and 0.9cm);
				\node [draw=none, fill=none] (S3) at (3.75, 1.2) {$M_6$};
				
				\node[fill = Goldenrod] (14) at (4.5, 0) {};
				\draw[dotted] (4.5, 0) ellipse (0.3cm and 0.3cm);
				\node [draw=none, fill=none] (I3) at (4.5, 0.7) {$M_7$};
                \begin{pgfonlayer}{bg}
				\foreach \i in {1,2}{
					\foreach \j in {3,4,5,6,7}{
						\draw[color=gray!30] (\i) -- (\j);	
					}			
				}
				\foreach \j in {4,5,6,7}{
					\draw[color=gray!30] (3) -- (\j);
				}
				\foreach \i in {4,5,6,7}{
					\foreach \j in {8,9,10,11}{
						\draw[color=gray!30] (\i) -- (\j);
					}
				}
				\foreach \i in {8,9,12,13,14}{
					\foreach \j in {10,11}{
						\draw[color=gray!30] (\i) -- (\j);
					}
				}
				\foreach \i in {12,13}{
					\draw[color=gray!30] (\i) -- (14);
				}
				\end{pgfonlayer}
				
  \end{tikzpicture}
  \end{center}
\end{subfigure}
\hfill
\begin{subfigure}{0.3\textwidth}
  \begin{center}
  \begin{tikzpicture}[every node/.style={inner sep=0pt, minimum size=4pt, circle, fill=black}]			
					\node[fill=blue,label=right:$M_1$] (S1) at (1,-2) {};
					\node[fill=Periwinkle,label=left:$M_2$] (I1) at (0,-1.5) {};
					\node[fill=WildStrawberry,label=right:$M_3$] (S2) at (1,-1) {};
					\node[fill=BrickRed,label=left:$M_4$] (I2) at (0,-.5) {};
					\node[fill=JungleGreen,label=right:$M_5$] (S3) at (1,0) {};
					\node[fill=YellowGreen,label=left:$M_6$] (I3) at (0,0.5) {};
					\node[fill=Goldenrod,label=right:$M_7$] (S4) at (1,1) {};
					
					\draw[color=gray!30] (S1) -- (I1) -- (S2) -- (S1);
					\draw[color=gray!30] (S2) -- (I2) -- (S3) -- (S2);
					\draw[color=gray!30] (S3) -- (I3) -- (S4) -- (S3);		
					
					\node[fill=none,draw=none] (blank) at (1,-1) {};		
  \end{tikzpicture}
  \end{center}
\end{subfigure}
\begin{subfigure}{\textwidth}
$$\lex(\cP) = 
	\prod_{s=1}^{2D+1} \lex(M_s)
	\sum_{\substack{0\le i_x\le m_x\\x=1,3,\ldots, 2D-1}}
		\prod_{d=1}^{D-1} 
			\binom{m_{2d}+m_{2d-1}+i_{2d+1}-i_{2d-1}-1}{m_{2d-1}-i_{2d-1}}\binom{m_{2d}+i_{2d+1}}{m_{2d}} $$
\end{subfigure}\\[1em]
\begin{subfigure}{0.3\textwidth}
  \begin{center}
  \begin{tikzpicture}[every node/.style={circle, fill=black, radius=3pt, inner sep=0pt, minimum size=4pt}]
	\node[fill=white] (0) at (-1.5,0.5) {\textbf{C}};

	\node[fill=JungleGreen] (1) at (0,-2.2) {};
	\node[fill=JungleGreen] (2) at (0, -1.8) {};
	\node[fill=blue] (3) at (1, -2) {};
	\node[fill=Goldenrod] (4) at (0, -1.2) {};
	\node[fill=Goldenrod] (5) at (0, -0.8) {};
	\node[fill=Periwinkle] (6) at (0.8, -1) {};
	\node[fill=Periwinkle] (7) at (1, -1) {};
	\node[fill=Periwinkle] (8) at (1.2, -1) {};
	\node[fill=WildStrawberry] (9) at (-0.2, 0) {};
	\node[fill=WildStrawberry] (10) at (0.2, 0) {};
	\draw[color=gray!30] (1) -- (2) -- (4) -- (5) -- (9);
	\draw[color=gray!30] (5) -- (10);
	\draw[color=gray!30] (6) -- (3) -- (4);
	\draw[color=gray!30] (7) -- (3) -- (8);
  \end{tikzpicture}
  \end{center}
\end{subfigure}
\begin{subfigure}{0.3\textwidth}
  \begin{center}
  \begin{tikzpicture}[every node/.style={circle, fill=black, radius=3pt, inner sep=0pt, minimum size=4pt}]
	\node[fill=JungleGreen] (1) at (0.4,-2) {};
	\node[fill=JungleGreen] (2) at (0.6, -2) {};
	\node[fill=blue] (3) at (0, -1) {};
	\node[fill=Goldenrod] (4) at (0.4, 0) {};
	\node[fill=Goldenrod] (5) at (0.6, 0) {};
	\node[fill=Periwinkle] (6) at (1, -1.1) {};
	\node[fill=Periwinkle] (7) at (1.1, -0.9) {};
	\node[fill=Periwinkle] (8) at (0.9, -0.9) {};
	\node[fill=WildStrawberry] (9) at (1.4, 0) {};
	\node[fill=WildStrawberry] (10) at (1.6, 0) {};
	\begin{pgfonlayer}{bg}
	\draw[color=gray!30] (6) -- (7) -- (8) -- (6);
	\draw[color=gray!30] (9) -- (10);
	\draw[color=gray!30] (1) -- (3) -- (2);
	\draw[color=gray!30] (6) -- (1) -- (7);
	\draw[color=gray!30] (6) -- (2) -- (7);
	\draw[color=gray!30] (1) -- (8) -- (2);
	\draw[color=gray!30] (6) -- (9) -- (7);
	\draw[color=gray!30] (6) -- (10) -- (7);
	\draw[color=gray!30] (10) -- (8) -- (9);
	\draw[color=gray!30] (6) -- (4) -- (7);
	\draw[color=gray!30] (6) -- (5) -- (7);
	\draw[color=gray!30] (4) -- (8) -- (5);
	\end{pgfonlayer}
	\draw[dotted] (0.5, -2) ellipse (0.3cm and 0.2cm);
	\draw[dotted] (0, -1) ellipse (0.2cm and 0.2cm);
	\draw[dotted] (1, -1) ellipse (0.3cm and 0.3cm);
	\draw[dotted] (0.5, 0) ellipse (0.3cm and 0.2cm);
	\draw[dotted] (1.5, 0) ellipse (0.3cm and 0.2cm);
	\node [draw=none, fill=none] (S1) at (-0.2, -2) {$M_1$};
	\node [draw=none, fill=none] (S1) at (-0.6, -1) {$M_2$};
	\node [draw=none, fill=none] (S1) at (1.7, -1) {$M_3$};
	\node [draw=none, fill=none] (S1) at (-0.2, 0) {$M_4$};
	\node [draw=none, fill=none] (S1) at (2.2, 0) {$M_5$};
  \end{tikzpicture}
  \end{center}
\end{subfigure}
\begin{subfigure}{0.3\textwidth}
  \begin{center}
  \begin{tikzpicture}
    \node[circle, fill=JungleGreen, radius=3pt, inner sep=0pt, minimum size=4pt,label=left:$M_1$] (2) at (0,-2) {};
    \node[circle, fill=blue, radius=3pt, inner sep=0pt, minimum size=4pt,label=left:$M_2$] (1) at (-.5,-1) {};
    \node[circle, fill=Periwinkle, radius=3pt, inner sep=0pt, minimum size=4pt,label=right:$M_3$] (4) at (.5, -1) {};
    \node[circle, fill=Goldenrod, radius=3pt, inner sep=0pt, minimum size=4pt,label=left:$M_4$] (3) at (0, 0) {};
    \node[circle, fill=WildStrawberry, radius=3pt, inner sep=0pt, minimum size=4pt,label=right:$M_5$] (5) at (1, 0) {};
    \draw[color=gray!30] (1) -- (2) -- (4) -- (5);
    \draw[color=gray!30] (3) -- (4);
  \end{tikzpicture}
  \end{center}
\end{subfigure}
\begin{subfigure}{\textwidth}
	$$ \lex(\mathcal{P})=\prod_{s=1}^{2D+1}
 \lex(M_s) \sum_{\substack{0\le i_x\le m_x\\x=1,3,\ldots, 2D-1}}
	\prod_{d=1}^{D-1} \binom{m_{2d}+m_{2d-1}+i_{2d+1}-i_{2d-1}-1}{m_{2d-1}-i_{2d-1}} $$
\end{subfigure}
\caption{
\textbf{Examples of the 
main results.} \textbf{A:} The path, \textbf{B:} the necklace and \textbf{C:} the binary tree. For each: \emph{Left}: the poset $\cP$, \emph{Middle}: the incomparability graph $\IG_\cP$ with a modular partition, \emph{Right}: the skeleton graph of the modular partition, \emph{Below}: the formula for the number of linear extensions.}
\label{fig:main_results}
\end{figure}

We begin with the case where the skeleton is a path. 
 \begin{theorem}\label{thm:path}
Let $\cP$ be a poset and $\cM=\{M_1,M_2,\ldots,M_{N}\}$ be a modular partition of $\IG_\cP$, where 
the skeleton graph $\IG_\cP/\cM$ is a path such that $M_i$ and $M_{i+1}$ are connected, $\forall i=1,\ldots,N-1$, see Figure~\ref{fig:main_results}A. Then
  	$$\lex(\mathcal{P})= \prod_{s=1}^N \lex(M_s) \sum_{\substack{0\le i_x\le m_{x}\\ x=1,\ldots,N-1}} \prod_{d=0}^{N} \binom{m_d - i_d + i_{d+1} - 1}{m_d - i_d}$$
  	where $m_k=|M_k|$ and $i_{N}=m_{N}$.
 \end{theorem}

We continue with the case where the skeleton is a necklace of cliques. 
First we define a necklace, following the terminology of \cite{belardo10}.
\begin{definition}
An  \emph{$n$-clique} is a complete undirected graph and a \emph{necklace} is an undirected graph obtained by gluing a list of cliques $(c_1, c_2, ..., c_n)$, where a single vertex of $c_i$ is glued to a single vertex of $c_{i+1}$ for $1\leq i < n$, with $c_{i-1}$ and $c_{i+1}$ glued to different vertices of $c_i$.  A \emph{$3$-necklace} is a necklace where all cliques are $3$-cliques.
\end{definition}

\begin{theorem}\label{thm:necklace}
Let $\cP$ be a poset and $\cM=\{M_1,M_2,\ldots,M_{2D+1}\}$ be a modular partition of $\IG_\cP$, where 
 the skeleton graph $\IG_\cP/\cM$ is a necklace of 3-cliques, see Figure~\ref{fig:main_results}B. Number the modules such that the cliques join on the odd indexed modules.
Then
\[\lex(\cP) = 
	\prod_{s=1}^{2D+1} \lex(M_s)
	\sum_{\substack{0\le i_x\le m_x\\x=1,3,\ldots, 2D-1}}
		\prod_{d=1}^{D-1} 
			\binom{m_{2d}+m_{2d-1}+i_{2d+1}-i_{2d-1}-1}{m_{2d-1}-i_{2d-1}}\binom{m_{2d}+i_{2d+1}}{m_{2d}} \]
where $m_k=|M_k|$ and $i_{2D+1}=m_{2D+1}$.
\end{theorem}

\begin{remark}
Note that if an incomparability graph admits a modular partition whose skeleton is a necklace with each clique of size at least $3$,
then it also admits a modular partition whose skeleton is a $3$-necklace. Indeed,  given a necklace of cliques $(c_1,\ldots,c_k)$, for any clique $c_i$ we can group $|c_i|-2$ modules together, corresponding to the  ``non-gluing vertices’’ on the necklace.  This gives a coarser partition whose skeleton graph is a $3$-necklace. 
Counting the modular partitions of these coarser modules in terms of the original modules is straightforward using Lemma~\ref{lem:JSgen}, since they correspond to independent sets in the comparability graph.
Therefore, the restriction to cliques of size three is not substantive. 
\end{remark}

Finally, we consider the case where the skeleton is a full binary tree. 

\begin{definition}
A \emph{full binary tree}, is a rooted tree where every node has either $0$ or $2$ children.  The path distance from the root to a node is its \emph{depth} and the maximal depth across all nodes is the \emph{total depth} of the tree.
\end{definition}

\begin{theorem}\label{thm:tree}
Let $\cP$ be a poset and $\cM=\{M_1,M_2,\ldots,M_{2D+1}\}$ be a modular partition of $\IG_\cP$, where 
the skeleton graph $\IG_\cP/\cM$ is a full binary tree of total depth $D$, see Figure~\ref{fig:main_results}C. Number the modules such that the even index modules are leaves, the last module $M_{2D+1}$ is also a leaf, and 
the modules $M_{2d}$ and~$M_{2d+1}$ are at depth $d$ in the tree.
Then
	$$ \lex(\mathcal{P})=\prod_{s=1}^{2D+1}
 \lex(M_s) \sum_{\substack{0\le i_x\le m_x\\x=1,3,\ldots, 2D-1}}
	\prod_{d=1}^{D-1} \binom{m_{2d}+m_{2d-1}+i_{2d+1}-i_{2d-1}-1}{m_{2d-1}-i_{2d-1}} $$
	where $m_k=|M_k|$ and $i_{2D+1}=m_{2D+1}$.
\end{theorem}

\begin{remark}
In \cite{Hab87} the \emph{decomposition diameter} of $\cP$ is defined as the size of the largest module of the prime modular partition, this is also known as \emph{modular width}.
The above are all posets with bounded decomposition diameter, thus a polynomial time algorithm exists for computing the number of linear extensions \cite[Proposition 3.4]{Hab87}. 
 We extend the work of \cite{Hab87} by providing closed formulas for the number of linear extensions for posets which are the lexicographic sum $\oplus_\cS \Upsilon$, where~$\Upsilon$
 is a path, 3-necklace or tree; or, as we see in the next section, when~$\Upsilon$ is locally a path, 3-necklace or tree.
\end{remark}

\section{Counting linear extensions via modular partitions of the incomparability graph}
\subsection{Main result and proofs}

In this section we prove our main result based on the structure of the skeleton of the incomparability graph. Intuitively, this structure may be viewed as being locally either a necklace of $3$-cliques, a tree, or a path. As a consequence, all results of Section~\ref{sec:necktreepath} follow from this result.  We achieve this by reducing the problem to counting the number of linear extension of the individual modules, and how they interact.  The following definition and lemma introduce the elements to make this precise.

\begin{definition}
Let $\cM=\{ M_i\}_{i=1}^k$ be a modular partition of a poset $\cP$.
Define $\mathcal{P}_\cM$ as the poset obtained by replacing every module of $\cP$ with a chain.  That is 
\[\cP_\cM := \oplus_{\cP/\cM} \Upsilon 
\quad\quad \text{where} \quad\quad 
\Upsilon=\{\mathbf{m_i}\}_{v_i\in \cP/\cM}
\quad\quad \text{and} \quad\quad m_i=|M_i|.\]

An \emph{$\cM$-linear extension of $\cP$} is an order-preserving bijection $\ell_\cM:\cP\to\cP_\cM$, which for all $i$ sends the module $M_i$ to the totally ordered module $\mathbf{m_i}$. 
Note in particular that if $\cM$ is the trivial partition consisting of all of $\cP$ then this restricts to the definition of a linear extension of $\cP$.
\end{definition}

The following lemma allows us to decompose the computation of the number of linear extensions using the modular partition. 
 This is analogous to a formula in the proof of \cite[Proposition 3.4]{Hab87}.
\begin{lemma}\label{lem:JSgen}
If  $\cM=\{ M_i\}_{i=1}^k$ is a modular partition of a poset $\cP$, then
  \[\lex(\cP)=\lex(\cP_\cM)\cdot\prod_{i=1}^{k}\lex(M_i).\]
\end{lemma}

\begin{proof}
Note that any linear extension $\ell:\cP\to\bn$ can be decomposed into an $\cM$-linear extension followed by a linear extension of $\cP_\cM$.  
Each factor in the formula accounts for the independent choices made at the corresponding stage of the partition.
\end{proof}

Since our result based on the structure of the incomparability graph, we recall the following known result:
\begin{lemma}\label{lem:sameIG}\cite{Edel89}
If two posets $\cP$ and $\cQ$ have the same incomparability graph, then $\lex(\cP)=\lex(\cQ)$.
\end{lemma}
Note the above result was originally stated in terms of the comparability graph, but since the comparability and incomparability graph are the graph complement of each other, it can be stated in terms of either.

Moreover, recall that a modular partition of a comparability graph needs not to be a modular partition of the poset, see Section~\ref{sec:notposet}.
However, if the quotient graph of the modular partition is the comparability graph of a poset, it is always possible to build a new poset with the same comparability graph and thus with the same number of linear extensions by Lemma~\ref{lem:sameIG}.  This can be made explicit with the following lemma.

\begin{lemma}\label{lem:reorient2}
Let $(\cP, \leq)$ be a poset and $\cM=\{M_1,\ldots,M_k\}$ be a modular partition of $\IG_\cP$.
If $\CG_\cP/\cM$ is transitively orientable, then there is a poset $(\cP, \sqsubseteq)$ that has the same incomparability graph as $(\cP, \leq)$ and for which $\cM$ is a poset partition and the order of the modules is a linear extension of $(\cP/\cM, \sqsubseteq)$, that is, if $i<j$ then $M_i\sqsubseteq M_j$.
\begin{proof}
Let $\cQ=\CG_\cP/\cM$, we can orient the edges in $\cQ$ such that $M_i\rightarrow M_j$ if and only if $i<j$. 
Since $\cM$ is a transitive orientation, $\cQ$ is a poset. Let $(\cP, \sqsubseteq)= \oplus_\cQ \{(M_i, \leq)\}_{M_i\in\cM}$.
Note that by construction the ordering of the modules is a linear extension of $\cQ$, thus by the definition of the lexicographic sum
each $M_i$ is a poset module.
\end{proof}
\end{lemma}

The argument underlying our main proof is similar to that used in \cite[Proposition 3.4]{Hab87}, where it is shown that the number of linear extensions of posets with bounded decomposition diameter can be computed in polynomial time. We use this approach to give formulas for certain posets and as our proof is constructive, we can obtain the set of all linear extensions, not just the number.
Intuitively, given a poset and a linear extension $\ell:\cP\to [n]$, one may subdivide $[n]$  into consecutive intervals. To compute $\lex(\cP)$, we first determine all admissible partitions of $[n]$  compatible with the structure of $\cP$, and then, for each possible partition, we count the number of local linear extensions on the resulting sub-intervals. This is formalised using the notion of a pivot, which we define below.

\begin{definition}[Pivot definition]\label{def:pivot}
Let $\cP$ be a poset and $\cM=\{M_r\}_{r=1}^{k}$ a modular partition of $\cP$ where
the numbering of the modules is a linear extension of $\cP/\cM$, i.e. if $\cM_s > \cM_r$  then $s>r$. 
Let~$\ell:\cP_\cM\to [n]$ be a linear extension.  
For each $1\leq r \leq k$ we define the 
\emph{$r$-th pivot element}, denoted~$e_r$, as the largest element of $M_r$, such that $\ell(e_r) < \ell(M_s)$ in $[n]$ for all $s>r$.  If there is no such element~$e_r=\emptyset$. 
We call $M_r^{first}$ the elements of $M_r$ that are less than or equal to the pivot in~$M_r$, and $M_r^{last}$ the elements larger than the pivot,  that is, 
\[M_r^{first} =  \{x\in M_r\,|\, x\leq e_r\}
\quad\quad \text{ and } \quad\quad
M_r^{last} =  \{x\in M_r\,|\, x > e_r\}.\]
\end{definition}

We now have all the elements to state our main result. 

\begin{theorem}\label{thm:joined}
Let $\cP$ be a poset and $\cM=\{M_1,M_2,\ldots,M_{2D+1}\}$ be a modular partition of the incomparability graph $\IG_\cP$. 
Some of the even modules may not exist in the decomposition, by abuse of notation we treat such modules as empty solely to simplify the indexing of the partition structure.
 If the following conditions on the skeleton graph $\IG_\cP/\cM$ hold: 
	\begin{itemize}
		\item there is an edge between $M_{2k-1}$ and $M_{2k+1}$; 
		\item if $M_{2k}$ is not empty there is an edge between $M_{2k-1}$ and $M_{2k}$; 
		\item there may or may not be an edge between $M_{2k}$ and $M_{2k+1}$;
		\item there are no other edges.
	\end{itemize}
Then
\[\lex(\cP) = 
\prod_{s=1}^{2D+1} \lex(M_s)
\sum_{(i_1, i_3, \ldots , i_{2D-1})}
\prod_{d=1}^{D-1} \binom{m_{2d}+m_{2d-1}+i_{2d+1}-i_{2d-1}-1}{m_{2d-1}-i_{2d-1}} \Phi_{2d}\]
where $m_k=|M_k|$,  the sum runs over tuples of all possible pivots, i.e.
\[ 0 \leq i_{2d-1} \leq m_{2d-1} \text{ for } 1\leq d<D  \quad\quad \text{and} \quad\quad i_{2D+1}=m_{2D+1}, \]
and 
\begin{equation*}
\Phi_{2d} =
\begin{cases}
\binom{m_{2d}+i_{2d+1}}{m_{2d}} & \text{ if } M_{2d} \neq \emptyset \text { and } M_{2d} \text{ and } M_{2d+1} \text{  are incomparable,} \\
1 														  & \text{ otherwise. }
\end{cases}
\end{equation*}
In the formulas we use the convention $\binom{-1}{0}=1$ and $\binom{n-1}{n}=0$ for $n>0$.
\end{theorem}

\begin{figure}
\begin{center}
\includegraphics[scale=0.69]{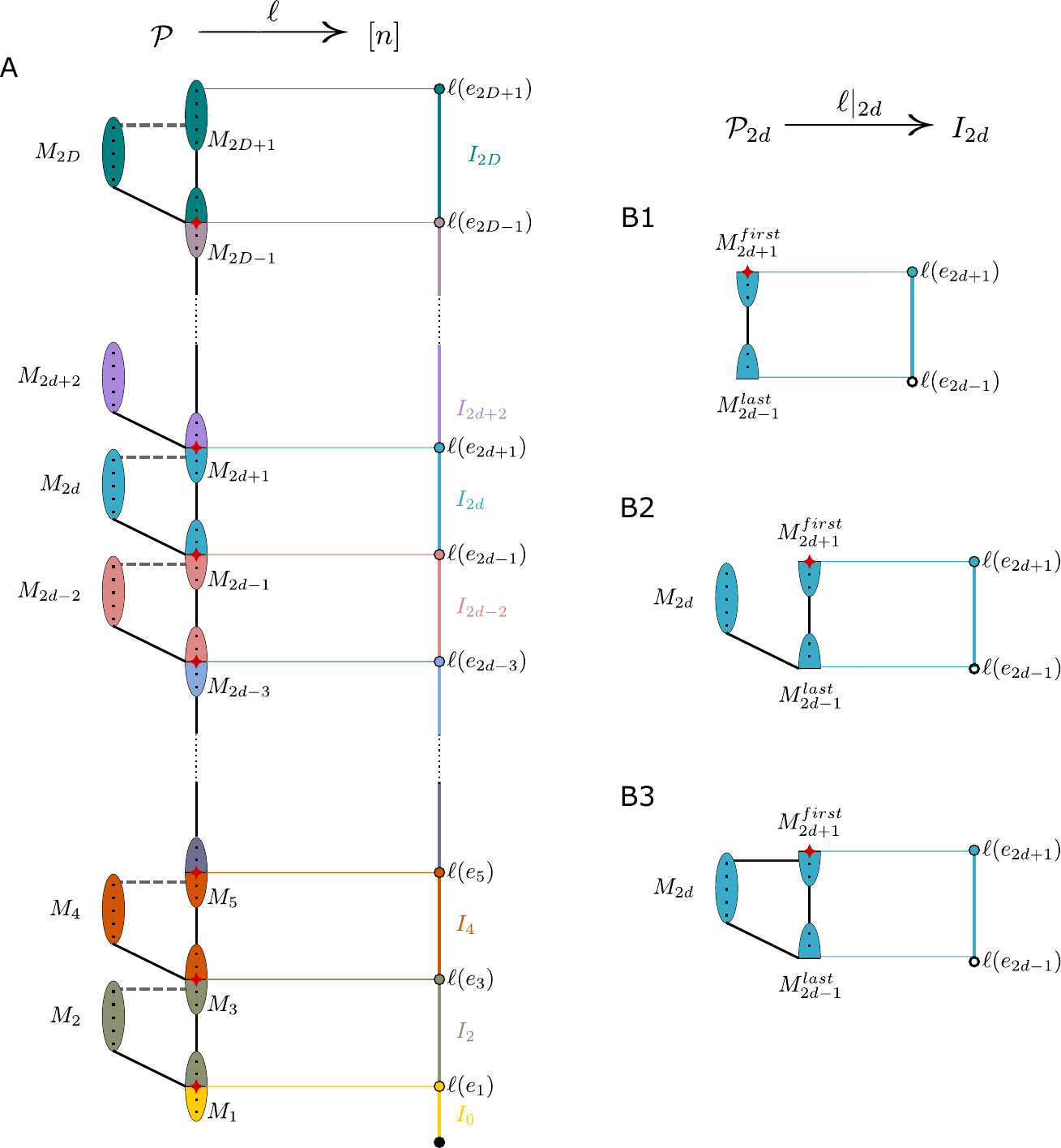}
\end{center}
\caption{\textbf{A:} Incomparability graph of the poset and linear extension.  Dash gray lines represent edges that may or may not be in the incomparability graph.  All odd modules are non-empty while even modules are allowed to be empty.  The red stars in the odd modules show the pivots giving the decomposition of $[n]$ into intervals. 
\textbf{B:}  Possible structures of the poset $\cP_{2d}$. }
\label{fig:poset_cartoon}
\end{figure}

\begin{proof}
Note that orienting the edges of the skeleton of the comparability graph $\CG_\cP/\cM$ according to the ordering of the modules
(i.e. $M_i\rightarrow M_j$ when $i<j$) is a transitive orientation.
To prove this we show that for any path $M_i\rightarrow M_j\rightarrow M_k$ in $\CG_\cP/\cM$, we must have $k>i+2$ so $M_i$ is comparable to $M_k$, thus $M_i\rightarrow M_k$.
To see that $k>i+2$, note that we cannot have $j=i+1$ and $k=i+2$, since $M_{2k-1}$ is always incomparable to $M_{2k}$ (if it exists, i.e. is non-empty), 
so if $i$ is odd we cannot have~$j=i+1$ and if $i$ is even then $i+1$ and $i+2$ must be incomparable, and if $M_{2k}$ is empty then~$k>i+2$ trivially.

So by Lemma~\ref{lem:reorient2}, if $\cM$ is not a poset partition or the order of the modules is not a linear extension of~$\cP$, we can create a new poset $(\cP,\sqsubseteq)$, 
for which these are true
and which has the same incomparability graph, thus has the same number of linear extensions by Lemma~\ref{lem:sameIG}. So we will count the linear extensions of $(\cP,\sqsubseteq)$ to get our result.
By Lemma \ref{lem:JSgen} it is enough to show that $\lex(\cP_{\cM})$  is given by the second part of the product. 
Thus,  to ease readability,  in the reminder of the proof we assume that $\cP_{\cM}=\cP$.

The idea of the proof is as follows.  
We first show that given the structure of the poset,  and a linear extension $\ell$ we can use the pivots of the odd modules to subdivide the poset $[n]$ into consecutive intervals $I_0<I_2<\cdots<I_{2D}$. 
To count $\lex(\cP)$ we determine all possible pivots for each of the odd modules.  For a fixed but arbitrary choice of pivots we count all possible local linear extensions 
\[\ell|_{2d}:l^{-1}(I_{2d}) \to I_{2d}\]
for $0\leq d \leq D$,  which we denote by $\lex({\cP_{2d}})$.
Then $\lex(\cP)$ is given by 
\[\lex(\cP)=\sum_{\text{possible pivots}} \prod_{d} \lex({\cP_{2d}}).\]

We start by determining the partition of $[n]$ into subintervals.
Let $\ell:\cP\to [n]$ be an arbitrary linear extension.  Given the structure of the incomparability graph, the definition of the pivot and the fact that the ordering of the modules is a linear extension of the skeleton of the comparability graph, we have that 
$M_{2d-1}^{first} < M_k$ for all $k>2d-1$.  Thus,  
\[\ell^{-1}([1, e_{2d+1}]) = M_{2d+1}^{first} \cup \bigcup_{k=1}^{2d}  M_k, \]
and 
\begin{equation} \label{h:tree}
\ell(e_{2d+1})=i_{2d+1}+\sum_{k=1}^{2d}m_k.
\end{equation}
Therefore,  the pivots of the odd modules split the poset $[n]$ into consecutive subintervals 
\begin{gather*}
I_0=[1 ,\ell(e1)],\quad\quad I_2=(\ell(e_1), \ell(e_3)],  \quad\quad\ldots \\I_{2D-2}=(\ell(e_{2D-3}),\ell(e_{2D-1})],\quad\quad I_{2D}=(\ell(e_{2D-1}),\ell(e_{2D+1})]=(\ell(e_{2D-1}),n],
\end{gather*}
for which 
\[\cP_0:=\ell^{-1}(I_0)=M_{1}^{first}, \]
\[\cP_{2d} :=\ell^{-1}(I_{2d})=M_{2d-1}^{last}\cup M_{2d}\cup M_{2d+1}^{first} \text{ for any } d\geq 1.\]
See Figure \ref{fig:poset_cartoon} for a sketch of this decomposition. 
Note in particular, that the number of elements in these intervals is given by 
\begin{equation}
|I_{2d}| =  \ell(e_{2d+1})-\ell(e_{2d-1}) = i_{2d+1}+\sum_{k=1}^{2d}m_k - i_{2d-1}+\sum_{k=1}^{2d-2}m_k = i_{2d+1}+ m_{2d}+ m_{2d-1} - i_{2d-1}.
\label{eq:interval_size}
\end{equation}\\

Given a fixed but arbitrary choice of pivots we now count the local linear extensions $l_{2d}:\cP_{2d}\to I_{2d}$.  
For $d=1$ we have 
\[\ell_0:M_1^{first} \to I_0.\]
Since $M_1$ is totally ordered,  there is only one choice for this poset map.  
For $d>1$ we count the possible linear extensions 
\[\ell|_{I_{2d}}:\cP_{2d}=M_{2d-1}^{last}\cup M_{2d}\cup M_{2d+1}^{first} \to I_{2d},\]
in two different cases.

In the first case,  either $M_{2d}$ is empty or $M_{2d}$ is not empty and $M_{2d}$ and $M_{2d+1}$ are comparable, corresponding to the local structure of the chain or the tree respectively (see Figure \ref{fig:poset_cartoon} B1 and B2).
Since all modules are linearly ordered and $ M_{2d} < M_{2d+1}$ if $ M_{2d}\neq\emptyset$,  then $\ell|_{2d}$ is completely determined once the images of the elements of $M_{2d-1}^{last}$ are determined. 
If $I_{2d}\neq \emptyset$, then by the definition of pivot an element of $M_{2d-1}^{last}$ cannot be mapped to its first element.
Thus,   by \eqref{eq:interval_size} 
\[\lex(\cP_{2d})
=\binom{|I_{2d}|-1}{|M_{2d-1}^{last}|}
=\binom{i_{2d+1}+m_{2d}+m_{2d-1}-i_{2d-1}-1}{m_{2d-1}-i_{2d-1}}.
\]
If $I_{2d}= \emptyset$,  there are no choices for its local linear extensions.  This coincides with the formula above, since by \eqref{eq:interval_size} in this case we have that: $m_{2d}=0$,  $m_{2d-1}=i_{2d-1}$ and $i_{2d+1}=0$. Thus, the coefficient obtained is $\binom{-1}{0}=1$.

In the second case,  $M_{2d}$ is not empty and $M_{2d}$ and $M_{2d+1}$ are incomparable, corresponding locally to the structure of a necklace  (see Figure \ref{fig:poset_cartoon}B3).
Since all modules are linearly ordered $\ell|_{2d}$ is completely determined once the images of the elements of $M_{2d-1}^{last}$ and $M_{2d}$ are determined. 
As before,  the first element $I_{2d}$ cannot have an element of $M_{2d-1}^{last}$ mapped to it.
Thus,  by \eqref{eq:interval_size} 
\begin{align*}
\lex({\cP_{2d}})
&=\binom{|I_{2d}|-1}{|M_{2d-1}^{last}|}\binom{|I_{2d}|-|M_{2d-1}^{last}|}{|M_{2d}|}\\
&=\binom{i_{2d+1}+m_{2d}+m_{2d-1}-i_{2d-1}-1}{m_{2d-1}-i_{2d-1}}\binom{i_{2d+1}+m_{2d}}{m_{2d}}.
\end{align*}\\

We finish the proof by determining the possible values for the pivots of the odd modules. 
By definition only the last element of the last module can be a pivot;
and for any $d < D$,  any element of~$M_{2d-1}$ can be a pivot,  since $M_{2d-1}$ and $M_{2d+1}$,  are incomparable and whenever  $M_{2d}\neq \emptyset$ also~$M_{2d-1}$ and $M_{2d}$ are incomparable. 
Therefore,  we have that all possible odd pivot indices are given by
\[ 0 \leq i_{2d-1} \leq m_{2d-1} \text{ for } 1\leq i<D  \quad\quad \text{and} \quad\quad i_{2D-1}=m_{2D-1}.\]

However,  it cannot hold that $m_{2d}=0$ and $i_{2d-1}=i_{2d+1}=0$.  We show this by contradiction. 
Assume it is true.  Since $i_{2d-1}=0$ and $m_{2d}=0$,  there is a $y\in M_{2d+1}$ such that $\ell(y)<\ell(M_{2d-1})$.  Similarly, since $i_{2d+1}=0$ there is an element $z\in M_{2d+2}\cup M_{2d+3}$ such that $\ell(z)<\ell(M_{2d+1})$. 
Thus we have that $\ell(z)<\ell(y)<\ell(M_{2d-1})$,  but this is not possible since $\ell(M_{2d-1})<\ell(M_{2d+2}\cup M_{2d+3})$.
Note however,  that if $m_{2d}=0$ and $i_{2d-1}=i_{2d+1}=0$, then 
\begin{align*}
\lex({\cP_{2d}})
&=\binom{0+0+m_{2d-1}-0-1}{m_{2d-1}}\\
&=\binom{m_{2d-1}-1}{m_{2d-1}}=0, 
\end{align*}
eliminating this choice of pivots from the sum.
On the other hand,  it is possible to have $m_{2d}>0$ and~$i_{2d-1}=i_{2d+1}=0$.
\end{proof}

The proofs of Theorems~\ref{thm:necklace}, \ref{thm:tree} and Theorem~\ref{thm:path} follow immediately from Theorem~\ref{thm:joined}.
The necklace (Theorem~\ref{thm:necklace}) follows from $\Phi_{2d}$ always being the top case, since for all $d$ we have that $M_{2d}$ and $M_{2d+1}$ are non-empty and incomparable. The tree (Theorem~\ref{thm:tree}) is the case when $M_{2d}$ and~$M_{2d+1}$ are always non-empty and comparable, thus we always have $\Phi_{2d}=1$. For Theorem~\ref{thm:path} we let all the even indexed modules be empty ($M_{2d}=\emptyset,\,\forall d$), and reindex the odd modules from $2d-1$ to~$d$. 
Indeed, we can view Theorem~\ref{thm:joined} as requiring that the skeleton graph is locally either a necklace, tree or path, thus is more general than the three results in Section~\ref{sec:necktreepath}.

\subsection{When is a partition of the incomparability graph not a poset partition}\label{sec:notposet}

Recall that a partition of the incomparability graph is a partition of the comparability graph and vice versa. Moreover, every modular partition of a poset induces a modular partition of both graphs; however, the converse does not hold in general. Indeed, the reverse implication fails when the elements of two modules are comparable in ``inconsistent'' ways. We make this precise below.

\begin{definition}
Consider a poset $\cP$ and a modular partition $\cM=\{M_1,\ldots,M_k\}$  of its comparability graph $\CG(\cP)$. The partition $\cM$ is not a poset partition if there are two distinct modules $M_i, M_j$ and~$x,y \in M_i$, $a,b\in M_j$ such that $x<a$ and $y>b$.  Note that it could be that either $x=y$ or $a=b$.
We call two such modules \emph{inconsistently comparable} and we denote this by $M_i\leftrightarrow M_j$ and we call the modules $M_i$ and $M_j$ \emph{inconsistent}.
We denote by $M_i-M_j$ if the two modules are comparable, either consistently or inconsistently and by $M_i\parallel M_j$ if they are incomparable.
\end{definition}

Since our results are phrased in terms of modular partitions of the incomparability graph rather than the poset itself, they are more general. However, a modular partition of the incomparability graph that satisfies the conditions of Theorem~\ref{thm:joined} is almost always also a modular partition of the poset. The only exceptions arise from the structure of the final two modules in the partition. We detail this in the following corollary 

\begin{corollary}\label{cor:the_ones}
Let $\cP$ be a poset and $\cM=\{M_1,M_2,\ldots,M_{2D+1}\}$ be a modular partition of the incomparability graph $\IG_\cP$ which satisfies the conditions of Theorem~\ref{thm:joined}. Then at least one of the following is true:
\begin{enumerate}
\item $\cM$ is a poset modular partition, or
\item $\{M_1,M_2,\ldots,M_{2D-1},M_{2D}\cup M_{2D+1}\}$ is a poset modular partition, or
\item $D=2$ and $\{M_1\cup M_2\cup M_4\cup M_5,M_3\}$ is a poset modular partition.
\end{enumerate}
\end{corollary}

In order to prove Corollary~\ref{cor:the_ones}, we first prove a more general result, for which the following lemma will be useful. 

\begin{lemma}\label{lemma:module_paths}
Consider a poset $\cP$ with a modular partition $\cM=\{M_1,\ldots,M_k\}$  of its comparability graph $\CG(\cP)$ that is not a modular partition of the poset.  Then the following hold: 
\begin{enumerate}
\item If $M_1 \leftrightarrow M_2 < M_3$ or $M_1 \leftrightarrow M_2 > M_3$ then $M_1-M_3$. 
\item If $M_1 \leftrightarrow M_2 - M_3-M_4$. Then either $M_1 - M_3$ or $M_2 - M_4$.
\end{enumerate}
\end{lemma}
\begin{proof}
Both claims follow by transitivity.  
To see the first case of $(1)$ note that there must be $x\in M_1$ and $a\in M_2$ such that $x<a$ and thus $x<M_3$.
The second case follows similarly.

To see $(2)$ holds, note that if $M_1 \parallel M_3$, then $\forall x\in M_2$ either 
\[M_1>x<M_3 \quad\quad\text{or}\quad\quad M_1<x>M_3.\]
Moreover, since $M_1\leftrightarrow M_2$ there must be $x,y\in M_2$ such that $M_1>x$ and $M_1<y$.  Thus, there is an~$a\in M_3$ such that $x<a<y$. 
Now let $b\in M_4$. Note that if $b<a$ then $b<y$ 
and if $a<b$ then~$x<b$. In either case, $M_2 - M_4$.
\end{proof}

We now use this to show the following, more general, proposition.

\begin{proposition}\label{prop:pos2mod}
Consider a poset $\cP$ and a modular partition $\cM=\{M_1,\ldots,M_k\}$  of its comparability graph $\CG(\cP)$ such that $\cM$ is not a poset partition. Then either:
 \begin{enumerate}
 \item there is a coarser modular partition $\widehat{\cM}=\{\widehat{M}_1,\ldots,\widehat{M}_t\}$, with $1<t<k$, where each $\widehat{M}_i$ is the union of one or more modules of $\cM$, or 
 \item there is a dominating vertex in the skeleton graph of the comparability graph $\CG(\cP)/\cM$, i.e. there is a module that is comparable to all others.
 \end{enumerate}
\end{proposition}

\begin{proof}
Since $\cM$ is not a partition of the poset, there are modules that are inconsistently comparable. The idea of the proof is to aggregate the modules that are inconsistently comparable to each other into single modules. Then show that this is a partition of the poset. Afterwards, we show that if condition 2) does not hold, this coarser partition has more than one module.

Let ${\sim}$ be the equivalence relation on $\{1, 2, \ldots, k\}$ generated by inconsistent comparability, i.e. 
the smallest equivalence relation such that
\[i\sim j \quad\quad \text{if} \quad\quad M_i\leftrightarrow M_j,\] 
and we denote by $K= \{1, \ldots, k\}/\sim$ the set of equivalence classes. 
We define a coarser partition 

\[\widehat{\cM}:=\lbrace\widehat{M}_{[i]}\rbrace_{[i]\in K}, 
\quad\quad \text{where}\quad\quad
\widehat{M}_{[i]}=\bigcup_{x\in [i]} M_x.
\]

First we show that $\widehat{\cM}$ is a modular partition of $\CG(\cP)$.
Consider $x\in [i]$ and $y\not\in [i]$ such that~$M_x-M_y$ are comparable.
Since $x$ and $y$ are not in the same equivalence class, then $M_x$ and~$M_y$ must be consistently comparable; without loss of generality suppose $M_x<M_y$.
We need to show~$M_z-M_y$ for all $z\in [i]$.
By definition, for any such $z$ there must be a chain of inconsistently comparable modules 
\[M_z\leftrightarrow M_{z_1} \leftrightarrow M_{z_2} \leftrightarrow M_{z_l} \leftrightarrow M_{x} < M_y.\]
By Lemma \ref{lemma:module_paths}(1), it follows that $M_{z_l}-M_y$, and since they are in different equivalence classes they must be consistently comparable. Iterating this argument we get that $M_z-M_y$ and thus $\widehat{M}_{[i]}$ is a module. 

If ${\sim}$ generates more than one equivalence class, then we are in part (1) of the statement.
On the other hand, if there is a single equivalence class, we show that there is a dominating vertex in the skeleton graph.
This is a consequence of the following claim:
 
\begin{claim*}
For every $[i]$, there is an $x\in[i]$ such that $M_x$ is comparable to $M_y$ for every $y\in[i]$. That is, the vertex induced subgraph of $\CG(\cP)/\cM$ on the vertices corresponding to $[i]$ has a dominating vertex in that subgraph. 
\begin{proof}
If $|[i]|\leq 2$ the result is trivial, suppose it is true for $|[i]|<k$ and let $|[i]|=k$.

For a contradiction suppose there is no $x\in [i]$ such that $M_x$ is comparable to every other module in $\widehat{M}_{[i]}$. 
Consider the vertex induced subgraph of the skeleton of the comparability graph $\CG(\cP)/\cM$ induced by the modules of $[i]$, i.e. the graph whose vertices are the elements of $[i]$.

Pick any vertex of this subgraph that is not a cut vertex, that is, a vertex whose removal does not disconnect the subgraph. 
Call this $M_{a_1}$. Since by construction the graph has no self loops and is connected, such a vertex exists \cite[Proposition 1.2.29]{west2001introduction}.
By the inductive hypothesis there is another vertex, we denote $M_{a_3}$, such that $M_{a_3}$ is comparable to every other module.  By construction there should be at least one module in $\widehat{M}_{[i]}$ inconsistently comparable to $M_{a_1}$, it must be different from $M_{a_3}$ otherwise $M_{a_3}$ is comparable to everything.  We denote it $M_{a_2}$.  Finally, there must be a module that is incomparable to $M_{a_2}$, otherwise $M_{a_2}$ is comparable to everything, we denote this module $M_{a_4}$. 
Thus we have constructed a path or cycle
\[M_{a_1}\leftrightarrow M_{a_2}-M_{a_3} - M_{a_4}, \]
where $M_{a_1}$ and $M_{a_3}$ are incomparable and $M_{a_2}$ and $M_{a_4}$ are incomparable.
However, by Lemma~\ref{lemma:module_paths}(2), this is not possible.  Thus, there must be an $x\in[i]$ such that $M_x$ is comparable to every other module in $\widehat{M}_{[i]}$.
\end{proof}
\end{claim*}
So if $\sim$ has a single equivalence class, then there is some module in $\cM$ that is comparable to all other modules.
\end{proof}

\begin{proof}[Proof of Corollary \ref{cor:the_ones}]
Consider a poset $\cP$ and modular partition $\cM=\{M_1,M_2,\ldots,M_{2D+1}\}$ of the incomparability graph, thus also of the comparability graph. If $\cM$ is a poset modular partition, then condition (1) is satisfied and we are done. 

Suppose $\cM$ is not a poset modular partition. Then by Lemma~\ref{prop:pos2mod} either there is a coarser modular partition or a dominating vertex in the skeleton graph $\CG(\cP)/\cM$. This dominating vertex would correspond to an isolated vertex in the skeleton of the incomparability graph $\IG(\cP)/\cM$.  However,  by the conditions of Theorem~\ref{thm:joined}, the skeleton graph $\IG(\cP)/\cM$ is connected, thus it cannot have an isolated vertex.

So by Lemma~\ref{prop:pos2mod}, if $\cM$ is not a poset modular partition, then there is a coarser modular partition of the poset with more than one module. 
However, to obtain a coarser partition one can only join together modules that have the same neighbourhood. 
Otherwise, the coarser partition would not be a modular partition.
If $D>2$ then the only such modules in the construction of Theorem~\ref{thm:joined} are~$M_{2D}$ and $M_{2D+1}$, which are both neighbours of $M_{2D-1}$. Taking the union of $M_{2D}$ and $M_{2D+1}$ we get a partition which does not have a coarser partition and does not have an isolated vertex in the incomparability graph of the skeleton, thus by the contrapositive of Lemma~\ref{prop:pos2mod} the partition is a poset partition.
If $D=2$, then $M_3$ is a neighbour of $M_1,M_2,M_4,M_5$, so we can join these together to get $\widehat{\cM}=\{M_1\cup M_2\cup M_4\cup M_5,M_3\}$, and again by the contrapositive of Lemma~\ref{prop:pos2mod}, $\widehat{\cM}$ is a poset partition.
\end{proof}

\section{Applications of linear extensions}\label{sec:app}
Our main result addresses the question: \textit{How many linear extensions does a poset have?} In this section, we restate this question in equivalent forms 
and discuss how this perspective shift can lead to interesting applications both within mathematics and in the sciences.

\subsection{Equivalent counting problems}

First recall that the problem of counting linear extensions is equivalent to counting permutations with certain constraints.

\begin{definition}
A permutation of length $n$ (or $n$-permutation) is an ordering of the set $[n]$.
Alternatively, an $n$-permutation $\pi=\pi_1\pi_2\ldots \pi_n$, can be thought of as a bijection $\pi:[n]\to[n]$ where $\pi(i)=\pi_i$ for any $1\leq i \leq n$.
For the $n$-permutation $\pi$ above, we say $\pi_i$ is to the left of $\pi_j$ (and $\pi_j$ is to the right of $\pi_i$),  if $i<j$.
Finally, we denote by $\Sigma_n$ the set of permutations on $[n]$.
\end{definition}

Recall that an inversion of an $n$-permutation $\pi$ is  pair of elements that occur outside of their ``natural order''.  More precisely:
  
\begin{definition}
  An \emph{inversion} in a permutation $\pi=\pi_1\dots\pi_n$ is a pair $(\pi_i,\pi_j)$ with $i<j$ and $\pi_i>\pi_j$.
  The \emph{inversion set} $I(\pi)$ is the set of all inversions in $\pi$.
 \end{definition}
 \begin{example}
  Consider the permutation $\pi=13425$, the inversions in $\pi$ are $(3,2)$ and $(4,2)$, so~$I(\pi)=\{(3,2),(4,2)\}$.
 \end{example}

\begin{remark}
A linear extension of a poset $\cP=([n], \leq)$, say $\ell:[n]\to \bn$, defines a permutation $\pi_\ell := \ell(1)\ell(2)\ldots \ell(n)$. Since $\ell$ is a linear extension, if $(i,j)$ is an inversion of $\pi_\ell$, then $i$ and $j$ must be incomparable in $\cP$.  Thus, we have that counting the number of linear extensions $\lex(\cP)$ is equivalent to counting the number of $n$-permutations whose inversion set is a subset of the set of pairs of incomparable elements of $\cP$,  that is, 
\[\lex(\cP) = |\left\{\pi\in\Sigma_n\,|\,I(\pi)\subseteq \{(i,j)\,|\, i, j \text{ incomparable in } \cP\}\right\}|.\]
See Figure~\ref{fig:counting_3way_example}AC for an example.
\end{remark}

This equivalent reformulation of the counting problem has brought forth several applications.  In particular, estimating the complexity of counting linear extensions \cite{BW91, dittmer2018counting} as well as the development of algorithmic methods to uniformly sample linear extensions in an efficient way \cite{PR94,huber2025generating} and to estimate the number of linear extensions \cite{banks2010using}. Moreover, in the other direction linear extensions have been used to understand permutation statistics \cite{BjWa91} and count permutation patterns and related variants \cite{EN12,yakoubov2015pattern,cooper2016complexity}.

One can also reformulate this counting problem in the setting of digraphs. To do this, we recall the notion of an acyclic tournament, which is a special kind of transitive graph.

\begin{definition}[Tournaments]
A \emph{tournament} is a directed graph with exactly one directed edge between every pair of vertices. A tournament with $n$ vertices is called an \emph{$n$-tournament}. 
We denote by $T_n$ the \emph{transitive} $n$-tournament with edges $i\to j$ if and only if $i<j$.  See Figure \ref{fig:3_ways_basics} for examples.
\end{definition}

\begin{remark}
A tournament is transitive if and only if it is acyclic.  In fact, in some literature the definition of a transitive tournament, is a tournament that is acyclic.
Moreover, an acyclic tournament on $n$ vertices is also referred to as a fully directed $n$-clique, as a directed $(n-1)$-simplex, or a simplex of dimension $n-1$.
\end{remark}

\begin{figure}[h!]
\begin{center}
\begin{tikzpicture}[scale=1]
  \node (label1) at (-1,2.25){\textbf{A}};
  \node[circle, draw=black, scale=0.5] (1a) at (0,0){$1$};
  \node[circle, draw=black, scale=0.5] (2a) at (1,0){$2$};
  \draw[thick,->-] (1a) to[bend left] (2a);
  \draw[thick,->-] (2a) to[bend left] (1a);
  \node (label2) at (.5,1){\small $C_2(\emptyset)=1$};
  \node[circle, draw=black, scale=0.5] (1b) at (0,1.5){$1$};
  \node[circle, draw=black, scale=0.5] (2b) at (1,1.5){$2$};
  \draw[thick,->-] (1b) to (2b);
  \node (label3) at (.5,-.5){\small $C_2(R_2)=2$};
  
  \def\x{4}  
  \node (label4) at (-1+\x,2.25){\textbf{B}};
  \node[circle, draw=black, scale=0.5] (1c) at (0+\x,1.5){$1$};
  \node[circle, draw=black, scale=0.5] (2c) at (1+\x,1.5){$2$};
  \node[circle, draw=black, scale=0.5] (3c) at (0+\x,.5){$3$};
  \draw[thick,->-] (1c) to (2c);
  \draw[thick,->-] (2c) to[bend right] (1c);
  \draw[thick,->-] (1c) to (3c);
  \draw[thick,->-] (2c) to (3c);
  \node (label5) at (.5+\x,0){\small $C_3(R)=2$};
    \node[circle, draw=black, scale=0.5] (1d) at (.5+\x,-2){$1$};
  \node[circle, draw=black, scale=0.5] (2d) at (0+\x,-1){$2$};
  \node[circle, draw=black, scale=0.5] (3d) at (1+\x,-1){$3$};
  \draw[thick] (3d) -- (1d) -- (2d);
  \node (label6) at (.5+\x,-2.5){\small $P_3(R)$};
  \node (label6a) at (.5+\x,2.25){\small $R=\{(2,1)\}$};
  \node[circle, draw=black, scale=0.5] (1e) at (3+\x,1.5){$1$};
  \node[circle, draw=black, scale=0.5] (2e) at (4+\x,1.5){$2$};
  \node[circle, draw=black, scale=0.5] (3e) at (3+\x,.5){$3$};
  \draw[thick,->-] (1e) to (2e);
  \draw[thick,->-] (3e) to[bend left] (1e);
  \draw[thick,->-] (1e) to (3e);
  \draw[thick,->-] (2e) to (3e);
  \node (label7) at (3.5+\x,0){\small $C_3(R)=1$};
  \node[circle, draw=black, scale=0.5] (1f) at (3.5+\x,-2){$1$};
  \node[circle, draw=black, scale=0.5] (2f) at (3.5+\x,-1.5){$2$};
  \node[circle, draw=black, scale=0.5] (3f) at (3.5+\x,-1){$3$};
  \draw[thick] (1f) -- (2f) -- (3f);
  \node (label8) at (3.5+\x,-2.5){\small $P_3(R)$};
  \node (label8a) at (3.5+\x,2.25){\small $R=\{(3,1)\}$};
  
  \def\x{11}  
  \node (label9) at (-1+\x,2.25){\textbf{C}};
  \node[circle, draw=black, scale=0.5] (1g) at (-.25+\x,1.75){$1$};
  \node[circle, draw=black, scale=0.5] (2g) at (1+\x,1.75){$2$};
  \node[circle, draw=black, scale=0.5] (3g) at (-.25+\x,.5){$3$};
  \node[circle, draw=black, scale=0.5] (4g) at (1+\x,.5){$4$};
  \draw[->-] (1g) to (2g);
  \draw[line width=1.25pt,->-,Red] (1g) to (3g);
  \draw[->-] (1g) to (4g); 
  \draw[line width=1.25pt,->-,Red] (3g) to (2g);
  \draw[line width=1.25pt,->-,Red] (2g) to (4g);
  \draw[->-] (3g) to (4g);
  \node (label10) at (.375+\x,2.25){\small $T$};
  \node (label11) at (.375+\x,-.5){\small $\phi(T)=1342$};
\end{tikzpicture}
\end{center}
\caption{\textbf{A:} The transitive tournament $T_2$ and $T_2$ with a reverse edge added, and the values of $C_2$ the number edge-induced transitive 2-tournaments in each.
\textbf{B:} The transitive tournament $T_3$ with two different reverse edges added, with the corresponding values of $C_3(R)$ and posets $P_3(R)$ beneath.
\textbf{C:} The Hamiltonian path through the transitive tournament $T$ and corresponding permutation $\phi(T)$.  }
\label{fig:3_ways_basics}
\end{figure}
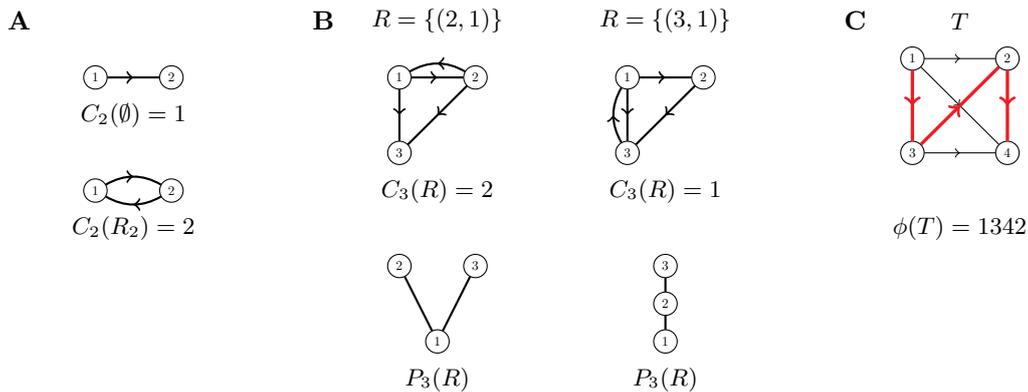


The tournament $T_n$ contains exactly one edge-induced transitive $n$-tournament, and $\binom{n}{k}$ edge-induced $k$-tournaments, since every edge-induced $k$-tournament in $T_n$ is equivalent to a vertex-induced subgraph of $T_n$.
However, if we take a directed graph with bidirectional connections then it is not sufficient to look only at the vertices or even at the number vertices and edges between them.  See Figure \ref{fig:3_ways_basics}AB for examples.

A natural question to ask is given the transitive tournament $T_n$ together with a set $R$ of reverse edges added to $T_n$, how many edge-induced transitive $n$-tournaments does the resulting graph contain?
We show that this can be equivalently formulated as counting the number of linear extensions of a certain subposet of $\bn$.
We start by formulating precisely this question in the digraph setting. 

\begin{definition}\label{def:reverse_edge_graph}
We denote by $R_n=\{(i,j)\in[n]\times[n]\,\mid \,i>j\}$, the set of \emph{all possible reverse edges} that can be added to $T_n$.
Given a chosen set $R\subseteq R_n$ of reverse edges we define $A_n(R)$ to be the graph with vertex set $[n]$ and edge set $E(T_n)\cup R$, that is, the graph obtained from $T_n$ by adding in the edges from $R$.
We denote by $\hat{C}_n(G)$ the number of edge-induced transitive $n$-tournaments in a graph $G$, and for convenience we define $C_n(R)\coloneqq\hat{C}_n(A_n(R))$. 
\end{definition}
 
So the question we aim to answer can be formulated as:
\begin{question}\label{que:1}
Given $n\in\mathbb{N}$ and a set of reverse edges $R$, can we find a formula for $C_n(R)$?
\end{question}

We now define a poset that asks this question in terms of linear extensions. 

\begin{definition}
Given $n\in\mathbb{N}$ and a set of reverse edges $R$, let $\cP_n(R)=([n], \leq_{\cP})$ be the poset with the order relation generated by 
\[i<_{\cP} j \text{ if } i<j \text{ and } (j,i)\notin R.\]
See Figure \ref{fig:3_ways_basics}B for examples. Note that the poset generated by a relation is equivalent to taking the transitive closure of that relation.
\end{definition}

\begin{proposition}\label{prop:posets_and_graphs}
The number of linear extensions of $\cP_n(R)$ is the number of edge-induced tournaments of $A_n(R)$.  That is 
\[\lex(\cP_n(R)) = C_n(R).\]
\end{proposition}
See Figure~\ref{fig:counting_3way_example}AB for an example of Proposition~\ref{prop:posets_and_graphs}. 
To show this we briefly recall the following graph theoretic result.

\begin{lemma_definition}\label{lem:HamPath}\cite[Theorem 1.5.2]{Iyer16}
A path in a graph $G$ is \emph{Hamiltonian} if it visits each vertex of $G$ exactly once.
A tournament has a unique Hamiltonian path if and only if the tournament is transitive.
This defines a bijection $\phi$ 
\[\begin{array}{c c c l}
\phi: & \left\{  \begin{array}{c}
						\text{Transitive tournaments}\\ 
						\text{with vertex set }[n] 
						\end{array}\right\}							& \to 			& \Sigma_n: \\
						&&&\\
        & T                                                      & \mapsto 	& \pi_T:[n]\to[n]
\end{array}\]
where  $\pi_T(1) \to \pi_T(2) \to \cdots \to \pi_T(n)$ is the unique Hamiltonian path of $T$.  See Figure \ref{fig:3_ways_basics}C for an example.
\end{lemma_definition}

Note that indeed $\phi$ is a bijection, with inverse $\phi^{-1}(\pi)$ the tournament with edges $(i,j)$, whenever~$\pi_i$ is left of $\pi_j$. Thus, transitive tournaments can equivalently  be thought of as $n$-permutations or totally ordered posets on the set $[n]$.
Notice moreover, that under this equivalence $T_n$ corresponds to the identity permutation $\pi = 12\ldots n$ and the poset $\bn$ respectively.

\begin{proof}[Proof of Proposition \ref{prop:posets_and_graphs}]

We show this by defining a bijection $\xi$ 
\[\begin{array}{c c c c}
\xi: & \left\{  \begin{array}{c}
						T \subseteq A_n(R)\\
						T \text{ is an edge induced transitive tournament} 
						\end{array}
		\right\}							
		& \to 			& 
		\left\{  \begin{array}{c}
						\ell:\cP_n(R)\to [n]\\
						\ell \text{ is a linear extension} 
						\end{array}
		\right\}:						 \\
						&&&\\
        & T & \mapsto 	& \xi(T) = \pi_T,
\end{array}\]
where $\pi_T$ is the permutation defined in Lemma \ref{lem:HamPath}.

We show first that $\pi_T$ is a linear extension by showing that if $i<_{\cP}j$, then there is an edge $i\to j$ in T and thus $\pi_T(i)< \pi_T(j)$ in $\bn$.
To see this, note first that if two vertices in $A_n(R)$ are connected by a unidirectional edge, then that edge must also be in $T$. 
Then note that if $i<_{\cP}j$ then either: $i<j$ in $[n]$ and $(j,i)\notin R$ or $i<_{\cP}j$ is generated by transitivity.  In the first case, $i$ and $j$ are connected by unidirectional edges, thus $i\to j$ must be an edge in $T$.
In the second case, there must be a chain of unidirectional edges in $A_n(R)$ 
\[i\to x_1 \to \ldots \to x_k \to j.\]
Then, $j\to i$ cannot be in $T$ since that will create a loop and $T$ is acyclic. Thus, $i\to j$ must be in $T$ instead. This shows $\xi(T)$ is a linear extension.
 
To see $\xi$  is bijective, consider the function that maps a linear extension $\ell$ to the tournament $T_\ell$ with Hamiltonian path $\ell(1) \to \ell(2) \to \cdots \to \ell(n)$.
Notice that if $i\to j$ is in $T_\ell$ then either $i<_\cP j$ or $i$ and $j$ are incomparable in $\cP_n(R)$.
In either case, $i\to j$ must be an edge in $A_n(R)$.  Thus, $T_\ell \subseteq A_n(R)$ and this function defines an inverse of $\xi$.  
\end{proof}
 
\begin{remark}
Note that equivalently, $C_n(R)$ can be counted as the number of $n$-permutations whose inversion set is contained in $R$.
Moreover, $C_n(R)$ can also be computed by counting the number of linear extensions of the poset defined by the transitive closure of the complement of $A_n(R)$, since that poset is the dual of $\cP_n(R)$ and thus has the same incomparability graph. 
\end{remark}

\begin{remark}
In \cite{Sid91} they introduce the notion of a flow on a network (later called Siderenko flows \cite{Cha23}), and show that the total flow of the Hasse diagram of a poset $P$ is equal to the number of linear extensions of $P$. They also introduce a function $\mu$ defined recursively on the cliques of an undirected graph, such that $\mu$ of the incomparability graph $P$ is equal to the number of linear extensions on $P$. This is another example of a property of graphs that is equivalent to counting linear extensions of poset, but distinct from that explored here.
\end{remark}
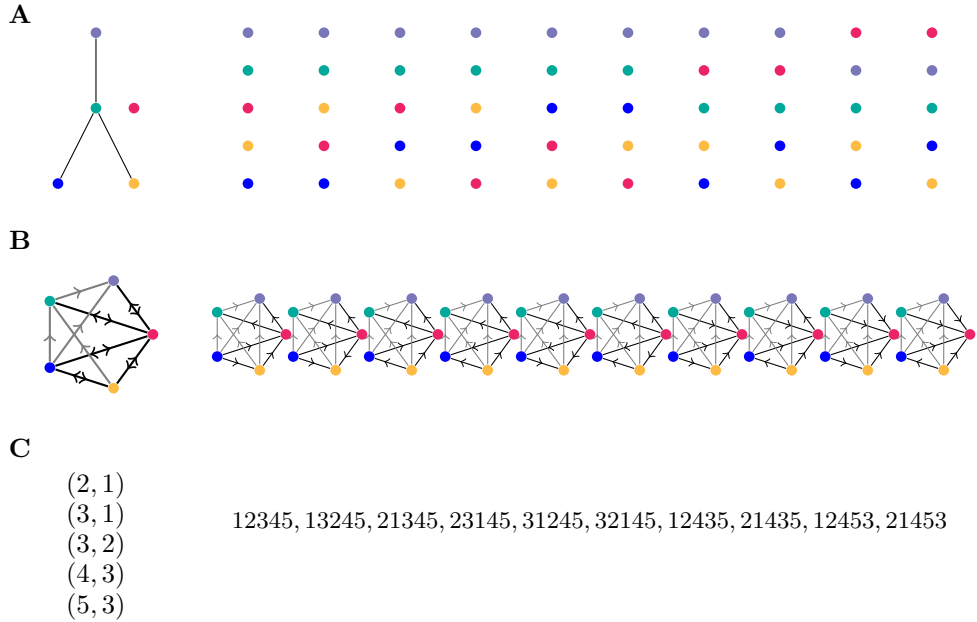
\begin{figure}[ht!]
\centering
\begin{tikzpicture}[scale=1]
 \def\x{-2}
   \node (label1) at (\x+-1,2.25){\textbf{A}};
   \node[circle, fill=black, radius=3pt, inner sep=0pt, minimum size=4pt, fill=blue] (1) at (-0.5+\x,0) {};
  \node[circle, fill=black, radius=3pt, inner sep=0pt, minimum size=4pt, fill=Dandelion] (2) at (0.5+\x,0) {};
  \node[circle, fill=black, radius=3pt, inner sep=0pt, minimum size=4pt, fill=WildStrawberry] (3) at (.5+\x, 1) {};
  \node[circle, fill=black, radius=3pt, inner sep=0pt, minimum size=4pt, fill=JungleGreen] (4) at (0+\x, 1) {};
  \node[circle, fill=black, radius=3pt, inner sep=0pt, minimum size=4pt, fill=Periwinkle] (5) at (0+\x, 2) {};
  \draw (5) -- (4) -- (2);
  \draw (4) -- (1);
 \def\ya{0}\def\yb{0.5}\def\yc{1}\def\yd{1.5}\def\ye{2}
   \def\x{0}
  \node[circle, fill=black, radius=3pt, inner sep=0pt, minimum size=4pt, fill=blue] (1) at (\x,\ya) {};
  \node[circle, fill=black, radius=3pt, inner sep=0pt, minimum size=4pt, fill=Dandelion] (2) at (\x,\yb) {};
  \node[circle, fill=black, radius=3pt, inner sep=0pt, minimum size=4pt, fill=WildStrawberry] (3) at (\x, \yc) {};
  \node[circle, fill=black, radius=3pt, inner sep=0pt, minimum size=4pt, fill=JungleGreen] (4) at (\x, \yd) {};
  \node[circle, fill=black, radius=3pt, inner sep=0pt, minimum size=4pt, fill=Periwinkle] (5) at (\x, \ye) {};
  \def\x{1}
  \node[circle, fill=black, radius=3pt, inner sep=0pt, minimum size=4pt, fill=blue] (1) at (\x,\ya) {};
  \node[circle, fill=black, radius=3pt, inner sep=0pt, minimum size=4pt, fill=Dandelion] (2) at (\x,\yc) {};
  \node[circle, fill=black, radius=3pt, inner sep=0pt, minimum size=4pt, fill=WildStrawberry] (3) at (\x, \yb) {};
  \node[circle, fill=black, radius=3pt, inner sep=0pt, minimum size=4pt, fill=JungleGreen] (4) at (\x, \yd) {};
  \node[circle, fill=black, radius=3pt, inner sep=0pt, minimum size=4pt, fill=Periwinkle] (5) at (\x, \ye) {};
  \def\x{2}
  \node[circle, fill=black, radius=3pt, inner sep=0pt, minimum size=4pt, fill=blue] (1) at (\x,\yb) {};
  \node[circle, fill=black, radius=3pt, inner sep=0pt, minimum size=4pt, fill=Dandelion] (2) at (\x,\ya) {};
  \node[circle, fill=black, radius=3pt, inner sep=0pt, minimum size=4pt, fill=WildStrawberry] (3) at (\x, \yc) {};
  \node[circle, fill=black, radius=3pt, inner sep=0pt, minimum size=4pt, fill=JungleGreen] (4) at (\x, \yd) {};
  \node[circle, fill=black, radius=3pt, inner sep=0pt, minimum size=4pt, fill=Periwinkle] (5) at (\x, \ye) {};
  \def\x{3}
  \node[circle, fill=black, radius=3pt, inner sep=0pt, minimum size=4pt, fill=blue] (1) at (\x,\yb) {};
  \node[circle, fill=black, radius=3pt, inner sep=0pt, minimum size=4pt, fill=Dandelion] (2) at (\x,\yc) {};
  \node[circle, fill=black, radius=3pt, inner sep=0pt, minimum size=4pt, fill=WildStrawberry] (3) at (\x, \ya) {};
  \node[circle, fill=black, radius=3pt, inner sep=0pt, minimum size=4pt, fill=JungleGreen] (4) at (\x, \yd) {};
  \node[circle, fill=black, radius=3pt, inner sep=0pt, minimum size=4pt, fill=Periwinkle] (5) at (\x, \ye) {};
  \def\x{4}
  \node[circle, fill=black, radius=3pt, inner sep=0pt, minimum size=4pt, fill=blue] (1) at (\x,\yc) {};
  \node[circle, fill=black, radius=3pt, inner sep=0pt, minimum size=4pt, fill=Dandelion] (2) at (\x,\ya) {};
  \node[circle, fill=black, radius=3pt, inner sep=0pt, minimum size=4pt, fill=WildStrawberry] (3) at (\x, \yb) {};
  \node[circle, fill=black, radius=3pt, inner sep=0pt, minimum size=4pt, fill=JungleGreen] (4) at (\x, \yd) {};
  \node[circle, fill=black, radius=3pt, inner sep=0pt, minimum size=4pt, fill=Periwinkle] (5) at (\x, \ye) {};
  \def\x{5}
  \node[circle, fill=black, radius=3pt, inner sep=0pt, minimum size=4pt, fill=blue] (1) at (\x,\yc) {};
  \node[circle, fill=black, radius=3pt, inner sep=0pt, minimum size=4pt, fill=Dandelion] (2) at (\x,\yb) {};
  \node[circle, fill=black, radius=3pt, inner sep=0pt, minimum size=4pt, fill=WildStrawberry] (3) at (\x, \ya) {};
  \node[circle, fill=black, radius=3pt, inner sep=0pt, minimum size=4pt, fill=JungleGreen] (4) at (\x, \yd) {};
  \node[circle, fill=black, radius=3pt, inner sep=0pt, minimum size=4pt, fill=Periwinkle] (5) at (\x, \ye) {};  
  \def\x{6}
  \node[circle, fill=black, radius=3pt, inner sep=0pt, minimum size=4pt, fill=blue] (1) at (\x,\ya) {};
  \node[circle, fill=black, radius=3pt, inner sep=0pt, minimum size=4pt, fill=Dandelion] (2) at (\x,\yb) {};
  \node[circle, fill=black, radius=3pt, inner sep=0pt, minimum size=4pt, fill=WildStrawberry] (3) at (\x, \yd) {};
  \node[circle, fill=black, radius=3pt, inner sep=0pt, minimum size=4pt, fill=JungleGreen] (4) at (\x, \yc) {};
  \node[circle, fill=black, radius=3pt, inner sep=0pt, minimum size=4pt, fill=Periwinkle] (5) at (\x, \ye) {};
   \def\x{7}
  \node[circle, fill=black, radius=3pt, inner sep=0pt, minimum size=4pt, fill=blue] (1) at (\x,\yb) {};
  \node[circle, fill=black, radius=3pt, inner sep=0pt, minimum size=4pt, fill=Dandelion] (2) at (\x,\ya) {};
  \node[circle, fill=black, radius=3pt, inner sep=0pt, minimum size=4pt, fill=WildStrawberry] (3) at (\x, \yd) {};
  \node[circle, fill=black, radius=3pt, inner sep=0pt, minimum size=4pt, fill=JungleGreen] (4) at (\x, \yc) {};
  \node[circle, fill=black, radius=3pt, inner sep=0pt, minimum size=4pt, fill=Periwinkle] (5) at (\x, \ye) {};
   \def\x{8}
  \node[circle, fill=black, radius=3pt, inner sep=0pt, minimum size=4pt, fill=blue] (1) at (\x,\ya) {};
  \node[circle, fill=black, radius=3pt, inner sep=0pt, minimum size=4pt, fill=Dandelion] (2) at (\x,\yb) {};
  \node[circle, fill=black, radius=3pt, inner sep=0pt, minimum size=4pt, fill=WildStrawberry] (3) at (\x, \ye) {};
  \node[circle, fill=black, radius=3pt, inner sep=0pt, minimum size=4pt, fill=JungleGreen] (4) at (\x, \yc) {};
  \node[circle, fill=black, radius=3pt, inner sep=0pt, minimum size=4pt, fill=Periwinkle] (5) at (\x, \yd) {};
  \def\x{9}
  \node[circle, fill=black, radius=3pt, inner sep=0pt, minimum size=4pt, fill=blue] (1) at (\x,\yb) {};
  \node[circle, fill=black, radius=3pt, inner sep=0pt, minimum size=4pt, fill=Dandelion] (2) at (\x,\ya) {};
  \node[circle, fill=black, radius=3pt, inner sep=0pt, minimum size=4pt, fill=WildStrawberry] (3) at (\x, \ye) {};
  \node[circle, fill=black, radius=3pt, inner sep=0pt, minimum size=4pt, fill=JungleGreen] (4) at (\x, \yc) {};
  \node[circle, fill=black, radius=3pt, inner sep=0pt, minimum size=4pt, fill=Periwinkle] (5) at (\x, \yd) {};
  \def\x{5}
    \begin{scope}[shift={(-2,-2)}]
  \node (label2) at (-1,1.25){\textbf{B}};
  \node[circle, fill=black, radius=3pt, inner sep=0pt, minimum size=4pt,fill=WildStrawberry] (g3) at (0:.75cm){};
  \node[circle, fill=black, radius=3pt, inner sep=0pt, minimum size=4pt,fill=Periwinkle] (g5) at (72:.75cm){};
  \node[circle, fill=black, radius=3pt, inner sep=0pt, minimum size=4pt,fill=JungleGreen] (g4) at (144:.75cm){};
  \node[circle, fill=black, radius=3pt, inner sep=0pt, minimum size=4pt,fill=blue] (g1) at (216:.75cm){};
  \node[circle, fill=black, radius=3pt, inner sep=0pt, minimum size=4pt,fill=Dandelion] (g2) at (288:.75cm){};
  \draw[thick,-<>-] (g1) to (g2);
  \draw[thick,-<>-] (g2) to (g3);
  \draw[thick,-<>-] (g3) to (g4);
  \draw[thick,-<>-] (g1) to (g3);
  \draw[thick,-<>-] (g3) to (g5);
  \draw[thick,->-,gray] (g1) to (g4);
  \draw[thick,->-,gray] (g1) to (g5);
  \draw[thick,->-,gray] (g2) to (g4);
  \draw[thick,->-,gray] (g4) to (g5);
  \end{scope}
  \begin{scope}[shift={(0,-2)}]
  \node[circle, fill=black, radius=3pt, inner sep=0pt, minimum size=4pt,fill=WildStrawberry] (g3) at (0:.5cm){};
  \node[circle, fill=black, radius=3pt, inner sep=0pt, minimum size=4pt,fill=Periwinkle] (g5) at (72:.5cm){};
  \node[circle, fill=black, radius=3pt, inner sep=0pt, minimum size=4pt,fill=JungleGreen] (g4) at (144:.5cm){};
  \node[circle, fill=black, radius=3pt, inner sep=0pt, minimum size=4pt,fill=blue] (g1) at (216:.5cm){};
  \node[circle, fill=black, radius=3pt, inner sep=0pt, minimum size=4pt,fill=Dandelion] (g2) at (288:.5cm){};
  \draw[->-] (g1) to (g2);
  \draw[->-] (g2) to (g3);
  \draw[->-] (g3) to (g4);
  \draw[->-,gray] (g4) to (g5);
  \draw[->-] (g1) to (g3);
  \draw[->-] (g3) to (g5);
  \draw[->-,gray] (g1) to (g4);
  \draw[->-,gray] (g1) to (g5);
  \draw[->-,gray] (g2) to (g4);
  \draw[->-,gray] (g2) to (g5);
  \end{scope}
  \begin{scope}[shift={(1,-2)}]
  \node[circle, fill=black, radius=3pt, inner sep=0pt, minimum size=4pt,fill=WildStrawberry] (g3) at (0:.5cm){};
  \node[circle, fill=black, radius=3pt, inner sep=0pt, minimum size=4pt,fill=Periwinkle] (g5) at (72:.5cm){};
  \node[circle, fill=black, radius=3pt, inner sep=0pt, minimum size=4pt,fill=JungleGreen] (g4) at (144:.5cm){};
  \node[circle, fill=black, radius=3pt, inner sep=0pt, minimum size=4pt,fill=blue] (g1) at (216:.5cm){};
  \node[circle, fill=black, radius=3pt, inner sep=0pt, minimum size=4pt,fill=Dandelion] (g2) at (288:.5cm){};
  \draw[->-] (g1) to (g2);
  \draw[-<-] (g2) to (g3);
  \draw[->-] (g3) to (g4);
  \draw[->-,gray] (g4) to (g5);
  \draw[->-] (g1) to (g3);
  \draw[->-] (g3) to (g5);
  \draw[->-,gray] (g1) to (g4);
  \draw[->-,gray] (g1) to (g5);
  \draw[->-,gray] (g2) to (g4);
  \draw[->-,gray] (g2) to (g5);
  \end{scope}
    \begin{scope}[shift={(2,-2)}]
  \node[circle, fill=black, radius=3pt, inner sep=0pt, minimum size=4pt,fill=WildStrawberry] (g3) at (0:.5cm){};
  \node[circle, fill=black, radius=3pt, inner sep=0pt, minimum size=4pt,fill=Periwinkle] (g5) at (72:.5cm){};
  \node[circle, fill=black, radius=3pt, inner sep=0pt, minimum size=4pt,fill=JungleGreen] (g4) at (144:.5cm){};
  \node[circle, fill=black, radius=3pt, inner sep=0pt, minimum size=4pt,fill=blue] (g1) at (216:.5cm){};
  \node[circle, fill=black, radius=3pt, inner sep=0pt, minimum size=4pt,fill=Dandelion] (g2) at (288:.5cm){};
  \draw[-<-] (g1) to (g2);
  \draw[->-] (g2) to (g3);
  \draw[->-] (g3) to (g4);
  \draw[->-,gray] (g4) to (g5);
  \draw[->-] (g1) to (g3);
  \draw[->-] (g3) to (g5);
  \draw[->-,gray] (g1) to (g4);
  \draw[->-,gray] (g1) to (g5);
  \draw[->-,gray] (g2) to (g4);
  \draw[->-,gray] (g2) to (g5);
  \end{scope}
    \begin{scope}[shift={(3,-2)}]
  \node[circle, fill=black, radius=3pt, inner sep=0pt, minimum size=4pt,fill=WildStrawberry] (g3) at (0:.5cm){};
  \node[circle, fill=black, radius=3pt, inner sep=0pt, minimum size=4pt,fill=Periwinkle] (g5) at (72:.5cm){};
  \node[circle, fill=black, radius=3pt, inner sep=0pt, minimum size=4pt,fill=JungleGreen] (g4) at (144:.5cm){};
  \node[circle, fill=black, radius=3pt, inner sep=0pt, minimum size=4pt,fill=blue] (g1) at (216:.5cm){};
  \node[circle, fill=black, radius=3pt, inner sep=0pt, minimum size=4pt,fill=Dandelion] (g2) at (288:.5cm){};
  \draw[-<-] (g1) to (g2);
  \draw[->-] (g2) to (g3);
  \draw[->-] (g3) to (g4);
  \draw[->-,gray] (g4) to (g5);
  \draw[-<-] (g1) to (g3);
  \draw[->-] (g3) to (g5);
  \draw[->-,gray] (g1) to (g4);
  \draw[->-,gray] (g1) to (g5);
  \draw[->-,gray] (g2) to (g4);
  \draw[->-,gray] (g2) to (g5);
  \end{scope}
    \begin{scope}[shift={(4,-2)}]
  \node[circle, fill=black, radius=3pt, inner sep=0pt, minimum size=4pt,fill=WildStrawberry] (g3) at (0:.5cm){};
  \node[circle, fill=black, radius=3pt, inner sep=0pt, minimum size=4pt,fill=Periwinkle] (g5) at (72:.5cm){};
  \node[circle, fill=black, radius=3pt, inner sep=0pt, minimum size=4pt,fill=JungleGreen] (g4) at (144:.5cm){};
  \node[circle, fill=black, radius=3pt, inner sep=0pt, minimum size=4pt,fill=blue] (g1) at (216:.5cm){};
  \node[circle, fill=black, radius=3pt, inner sep=0pt, minimum size=4pt,fill=Dandelion] (g2) at (288:.5cm){};
  \draw[->-] (g1) to (g2);
  \draw[-<-] (g2) to (g3);
  \draw[->-] (g3) to (g4);
  \draw[->-,gray] (g4) to (g5);
  \draw[-<-] (g1) to (g3);
  \draw[->-] (g3) to (g5);
  \draw[->-,gray] (g1) to (g4);
  \draw[->-,gray] (g1) to (g5);
  \draw[->-,gray] (g2) to (g4);
  \draw[->-,gray] (g2) to (g5);
  \end{scope}
    \begin{scope}[shift={(5,-2)}]
  \node[circle, fill=black, radius=3pt, inner sep=0pt, minimum size=4pt,fill=WildStrawberry] (g3) at (0:.5cm){};
  \node[circle, fill=black, radius=3pt, inner sep=0pt, minimum size=4pt,fill=Periwinkle] (g5) at (72:.5cm){};
  \node[circle, fill=black, radius=3pt, inner sep=0pt, minimum size=4pt,fill=JungleGreen] (g4) at (144:.5cm){};
  \node[circle, fill=black, radius=3pt, inner sep=0pt, minimum size=4pt,fill=blue] (g1) at (216:.5cm){};
  \node[circle, fill=black, radius=3pt, inner sep=0pt, minimum size=4pt,fill=Dandelion] (g2) at (288:.5cm){};
  \draw[-<-] (g1) to (g2);
  \draw[-<-] (g2) to (g3);
  \draw[->-] (g3) to (g4);
  \draw[->-,gray] (g4) to (g5);
  \draw[-<-] (g1) to (g3);
  \draw[->-] (g3) to (g5);
  \draw[->-,gray] (g1) to (g4);
  \draw[->-,gray] (g1) to (g5);
  \draw[->-,gray] (g2) to (g4);
  \draw[->-,gray] (g2) to (g5);
  \end{scope}
    \begin{scope}[shift={(6,-2)}]
  \node[circle, fill=black, radius=3pt, inner sep=0pt, minimum size=4pt,fill=WildStrawberry] (g3) at (0:.5cm){};
  \node[circle, fill=black, radius=3pt, inner sep=0pt, minimum size=4pt,fill=Periwinkle] (g5) at (72:.5cm){};
  \node[circle, fill=black, radius=3pt, inner sep=0pt, minimum size=4pt,fill=JungleGreen] (g4) at (144:.5cm){};
  \node[circle, fill=black, radius=3pt, inner sep=0pt, minimum size=4pt,fill=blue] (g1) at (216:.5cm){};
  \node[circle, fill=black, radius=3pt, inner sep=0pt, minimum size=4pt,fill=Dandelion] (g2) at (288:.5cm){};
  \draw[->-] (g1) to (g2);
  \draw[->-] (g2) to (g3);
  \draw[-<-] (g3) to (g4);
  \draw[->-,gray] (g4) to (g5);
  \draw[->-] (g1) to (g3);
  \draw[->-] (g3) to (g5);
  \draw[->-,gray] (g1) to (g4);
  \draw[->-,gray] (g1) to (g5);
  \draw[->-,gray] (g2) to (g4);
  \draw[->-,gray] (g2) to (g5);
  \end{scope}
    \begin{scope}[shift={(7,-2)}]
  \node[circle, fill=black, radius=3pt, inner sep=0pt, minimum size=4pt,fill=WildStrawberry] (g3) at (0:.5cm){};
  \node[circle, fill=black, radius=3pt, inner sep=0pt, minimum size=4pt,fill=Periwinkle] (g5) at (72:.5cm){};
  \node[circle, fill=black, radius=3pt, inner sep=0pt, minimum size=4pt,fill=JungleGreen] (g4) at (144:.5cm){};
  \node[circle, fill=black, radius=3pt, inner sep=0pt, minimum size=4pt,fill=blue] (g1) at (216:.5cm){};
  \node[circle, fill=black, radius=3pt, inner sep=0pt, minimum size=4pt,fill=Dandelion] (g2) at (288:.5cm){};
  \draw[-<-] (g1) to (g2);
  \draw[->-] (g2) to (g3);
  \draw[-<-] (g3) to (g4);
  \draw[->-,gray] (g4) to (g5);
  \draw[->-] (g1) to (g3);
  \draw[->-] (g3) to (g5);
  \draw[->-,gray] (g1) to (g4);
  \draw[->-,gray] (g1) to (g5);
  \draw[->-,gray] (g2) to (g4);
  \draw[->-,gray] (g2) to (g5);
  \end{scope}
    \begin{scope}[shift={(8,-2)}]
  \node[circle, fill=black, radius=3pt, inner sep=0pt, minimum size=4pt,fill=WildStrawberry] (g3) at (0:.5cm){};
  \node[circle, fill=black, radius=3pt, inner sep=0pt, minimum size=4pt,fill=Periwinkle] (g5) at (72:.5cm){};
  \node[circle, fill=black, radius=3pt, inner sep=0pt, minimum size=4pt,fill=JungleGreen] (g4) at (144:.5cm){};
  \node[circle, fill=black, radius=3pt, inner sep=0pt, minimum size=4pt,fill=blue] (g1) at (216:.5cm){};
  \node[circle, fill=black, radius=3pt, inner sep=0pt, minimum size=4pt,fill=Dandelion] (g2) at (288:.5cm){};
  \draw[->-] (g1) to (g2);
  \draw[->-] (g2) to (g3);
  \draw[-<-] (g3) to (g4);
  \draw[->-,gray] (g4) to (g5);
  \draw[->-] (g1) to (g3);
  \draw[-<-] (g3) to (g5);
  \draw[->-,gray] (g1) to (g4);
  \draw[->-,gray] (g1) to (g5);
  \draw[->-,gray] (g2) to (g4);
  \draw[->-,gray] (g2) to (g5);
  \end{scope}
    \begin{scope}[shift={(9,-2)}]
  \node[circle, fill=black, radius=3pt, inner sep=0pt, minimum size=4pt,fill=WildStrawberry] (g3) at (0:.5cm){};
  \node[circle, fill=black, radius=3pt, inner sep=0pt, minimum size=4pt,fill=Periwinkle] (g5) at (72:.5cm){};
  \node[circle, fill=black, radius=3pt, inner sep=0pt, minimum size=4pt,fill=JungleGreen] (g4) at (144:.5cm){};
  \node[circle, fill=black, radius=3pt, inner sep=0pt, minimum size=4pt,fill=blue] (g1) at (216:.5cm){};
  \node[circle, fill=black, radius=3pt, inner sep=0pt, minimum size=4pt,fill=Dandelion] (g2) at (288:.5cm){};
  \draw[-<-] (g1) to (g2);
  \draw[->-] (g2) to (g3);
  \draw[-<-] (g3) to (g4);
  \draw[->-,gray] (g4) to (g5);
  \draw[->-] (g1) to (g3);
  \draw[-<-] (g3) to (g5);
  \draw[->-,gray] (g1) to (g4);
  \draw[->-,gray] (g1) to (g5);
  \draw[->-,gray] (g2) to (g4);
  \draw[->-,gray] (g2) to (g5);
  \end{scope}
  \def\y{-3}  \def\x{-2}
  \node (label3) at (\x+-1,\y-.5){\textbf{C}};
  \node[fill opacity=0,text opacity=1] (perms) at (\x,-1.0+\y) {$(2,1)$};
  \node[fill opacity=0,text opacity=1] (perms) at (\x,-1.4+\y) {$(3,1)$};
  \node[fill opacity=0,text opacity=1] (perms) at (\x,-1.8+\y) {$(3,2)$};
  \node[fill opacity=0,text opacity=1] (perms) at (\x,-2.2+\y) {$(4,3)$};
  \node[fill opacity=0,text opacity=1] (perms) at (\x,-2.6+\y) {$(5,3)$};
  \node[fill opacity=0,text opacity=1] (perms) at (4.5,-1.5+\y) {\small $12345,13245,21345,23145,31245,32145,12435,21435,12453,21453$};
 \end{tikzpicture}
 \caption{Let $R=\{(2,1),(3,1),(3,2),(4,3),(5,3)\}$. \textbf{A}: The poset $P_5(R)$ along with all its linear extensions. 
 \textbf{B}: The digraph $A_5(R)$ and all edge-induced transitive 5-tournaments it contains.
 \textbf{C}: The set $R$ and all permutations whose inversions set is a subset of $R$. The columns represent the bijection $\phi$ between the top and middle row and $\psi$ between the middle and bottom row.}\label{fig:counting_3way_example}
 \end{figure}
\begin{figure}[h]
\begin{center}
\includegraphics[scale=0.78]{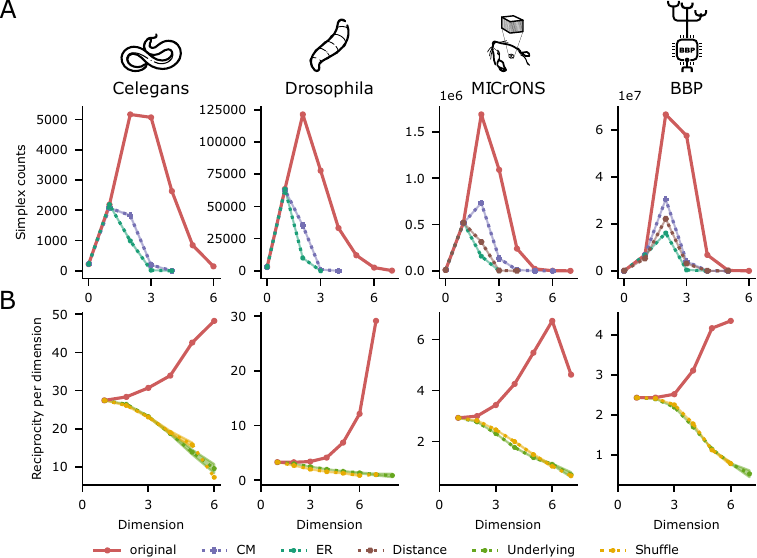}
\end{center}
\caption{\textbf{Overexpression of directed motifs} 
\textbf{A:} Simplex counts for each dimension for several networks at cellular resolution: the mature C. elegans, the Drosophila larva, an electron microscopic reconstruction of a small volume of mouse cortical connectivity (MICrONS) and a morphologically detailed in-silico model of a small volume of rat cortical connectivity.
These counts are contrasted to the ones of corresponding random controls of different types: Erd\H{o}s–R\'enyi, configuration model, and distance dependent graphs, which by design match the counts of the original network for dimensions~$0$ and $1$ (i.e. with the same number of vertices and edges). 
\textbf{B:} Percentage of reciprocal connections in the subgraph of simplices of each dimension contrasted with the same curve for controls where just the directionality of the connections is modified. 
For both panels, for the various controls, dots denote mean values in that dimension and shaded regions are the standard error of the mean. 
Reproduced from \cite{santander2025heterogeneous}, with permission from the authors.}
\label{fig:reliab_fig}
\end{figure}

\subsection{Applications to the sciences}

We have shown that counting linear extensions of the poset $\bn$ is equivalent to counting $n$-permutations with certain restrictions on their inversion set, or counting the number of transitive $n$-tournaments (or directed $n$-simplices) on digraphs built from adding reciprocal edges to a single transitive tournament (or directed $n$-simplex).
The latter is a key question in the sciences, particularly neuroscience, where the structure of a network shapes its function. The former can be characterised by the under- or over-expression of directed $n$-simplices. Specifically, across species and across scales, neural circuits are highly organised, and this organisation can be quantified using (directed) simplex counts \citep{sizemore2018cliques, tadic2019functional, sizemore2019importance, andjelkovic2020topology, shi2021computing}.

More precisely at the cellular scale, consider the directed graph built from a neural circuit, where the vertices represent neurons and there is an edge from neuron $a$ to neuron $b$ if there is at least one synaptic connection from $a$ to $b$.  It has been shown, that these networks are sparse, have long tailed degree distributions and have an over-expression of both reciprocal connections and directed simplices, particularly in the higher dimensions \cite{towlson2013rich, sizemore2019importance, reimann2024modeling, santander2025heterogeneous}.  See Figure~\ref{fig:reliab_fig}A for an example across species.
Moreover, \cite{santander2025heterogeneous} shows that the locations of reciprocal connections within the network are not random, but rather that they preferentially lie within simplices and that this preference increases with the dimension of the simplex, see Figure \ref{fig:reliab_fig}B.  Furthermore, \cite{BarrosHumanvRat} shows that the location of reciprocal connections drives a substantial part of the structure that differentiates human from rodent network connectivity.

However, given the effect of the addition of reciprocal connections on the number of directed $n$-simplices, one could wonder if such preference is merely a combinatorial artefact.
In \cite{santander2025heterogeneous} and~\cite{SantoroThesis}, random controls that shuffle the locations of reciprocal connections in the network are used to empirically show that this is not the case.
The question remains however, on how to provide theoretical guarantees.  
Thus, a natural question arises of how much of the higher dimensional structure of biological neural networks is driven by reciprocal connections and their location?
A related question is addressed in \cite{unger2024mcmc}, where they attempt to sample networks at random with approximately fixed directed simplex counts from the set of directed networks with the same underlying undirected network. They found indeed, that for certain architectures a large part of the structure can arise from the location of reciprocal connections. 
Question \ref{que:1} addresses the simplest version of this inquiry, by simplifying the underlying graph to its bare bones and setting it to be simply a directed $n$-simplex, rather than an arbitrary directed acyclic graph.

The translation of the question to the world of posets, shows that answering it analytically is $\#P$-complete.
However, it provides tools for uniformly sampling these structures, which can yield better numerical approximations of expected and unexpected behaviours.
Moreover, our formulas and related work could serve as building blocks for constructive or engineering methods to generate posets, and therefore networks, with certain global structure and specified local directed clique counts, by controlling the skeleton and the structure of the modular components. For example, suppose we require a directed graph consisting of four modules arranged in a $4$-cycle (see Figure~\ref{fig:construct_graph}), with each module having a prescribed size and a specified number of directed cliques. We can construct such a graph by using our formulas, together with existing results, to generate posets of the required size and number of linear extensions in each module, and then using their corresponding digraphs to build the required graph.

\begin{figure}[h]
\centering
\begin{tikzpicture}[scale=1]
\node (figC) at (-2.5,2.9){\small \textbf{A}};
  \node[circle, draw=black, radius=3pt, inner sep=2pt, text width=13mm, align=center] (a) at (-1,1){\small $n=5$ $C_n=10$};
  \node[circle, draw=black, radius=3pt, inner sep=2pt, text width=13mm, align=center] (b) at (1,1){\small $n=4$  $C_n=2$};
  \node[circle, draw=black, radius=3pt, inner sep=2pt, text width=13mm, align=center] (c) at (1,-1){\small $n=5$  $C_n=20$};
  \node[circle, draw=black, radius=3pt, inner sep=2pt, text width=13mm, align=center] (d) at (-1,-1){\small $n=3$, $C_n=6$};
  \node (A) at (-2,2){\small $W$};
  \node (B) at (2,2){\small $X$};
  \node (C) at (2,-2){\small $Y$};
  \node (D) at (-2,-2){\small $Z$};
  \draw[->-,thick] (a) -- (b);
  \draw[->-,thick] (b) -- (c);
  \draw[->-,thick] (c) -- (d);
  \draw[->-,thick] (d) -- (a);
  \node (dummy) at (0,-2.7){\mbox{}};
\end{tikzpicture}
\hfill
\begin{tikzpicture}[scale=1]
\node (figB) at (-2,3.2){\small \textbf{B}};
\begin{scope}[shift={(-1,.75)}]
  \node (1) at (0,1.5) {$P_5(R_W)$};
  \node[circle, fill=black, radius=3pt, inner sep=0pt, minimum size=4pt, fill=blue] (1) at (-0.5,0) {};
  \node[circle, fill=black, radius=3pt, inner sep=0pt, minimum size=4pt, fill=Dandelion] (2) at (0.5,0) {};
  \node[circle, fill=black, radius=3pt, inner sep=0pt, minimum size=4pt, fill=WildStrawberry] (3) at (.5, .5) {};
  \node[circle, fill=black, radius=3pt, inner sep=0pt, minimum size=4pt, fill=JungleGreen] (4) at (0, .5) {};
  \node[circle, fill=black, radius=3pt, inner sep=0pt, minimum size=4pt, fill=Periwinkle] (5) at (0, 1) {};
  \draw (5) -- (4) -- (2);
  \draw (4) -- (1);
\end{scope}
\begin{scope}[shift={(1,.75)}]
  \node (1) at (0,1.5) {$P_4(R_X)$};
  \node[circle, fill=black, radius=3pt, inner sep=0pt, minimum size=4pt, fill=Orchid] (1) at (-0.5,0) {};
  \node[circle, fill=black, radius=3pt, inner sep=0pt, minimum size=4pt, fill=Maroon] (2) at (0.5,0) {};
  \node[circle, fill=black, radius=3pt, inner sep=0pt, minimum size=4pt, fill=Cerulean] (4) at (0, .5) {};
  \node[circle, fill=black, radius=3pt, inner sep=0pt, minimum size=4pt, fill=NavyBlue] (5) at (0, 1) {};
  \draw (5) -- (4) -- (2);
  \draw (4) -- (1);
\end{scope}
\begin{scope}[shift={(1,-.75)}]
  \node (1) at (0.25,-1) {$P_5(R_Y)$};
  \node[circle, fill=black, radius=3pt, inner sep=0pt, minimum size=4pt, fill=Blue] (1) at (-0.5,0) {};
  \node[circle, fill=black, radius=3pt, inner sep=0pt, minimum size=4pt, fill=Goldenrod] (2) at (0.5,0) {};
  \node[circle, fill=black, radius=3pt, inner sep=0pt, minimum size=4pt, fill=Lavender] (3) at (1, .25) {};
  \node[circle, fill=black, radius=3pt, inner sep=0pt, minimum size=4pt, fill=BrickRed] (4) at (-0.5, .5) {};
  \node[circle, fill=black, radius=3pt, inner sep=0pt, minimum size=4pt, fill=CadetBlue] (5) at (0.5, .5) {};
  \draw (1) -- (4) -- (2) -- (5) -- (1);
\end{scope}
\begin{scope}[shift={(-1,-.75)}]
  \node (1) at (0,-1) {$P_3(R_Z)$};
  \node[circle, fill=black, radius=3pt, inner sep=0pt, minimum size=4pt, fill=Yellow] (1) at (-0.5,0.25) {};
  \node[circle, fill=black, radius=3pt, inner sep=0pt, minimum size=4pt, fill=Tan] (2) at (0,0.25) {};
  \node[circle, fill=black, radius=3pt, inner sep=0pt, minimum size=4pt, fill=YellowGreen] (3) at (.5, 0.25) {};
\end{scope}
\node (dummy) at (0,-2.4){\mbox{}};
\end{tikzpicture}
\hfill
\begin{tikzpicture}[scale=1]
\node (figC) at (-2.75,2.95){\small \textbf{C}};
\begin{scope}[shift={(-1.5,1.5)}]
  \node[circle, fill=black, radius=3pt, inner sep=0pt, minimum size=4pt,fill=WildStrawberry] (g3a) at (0:.75cm){};
  \node[circle, fill=black, radius=3pt, inner sep=0pt, minimum size=4pt,fill=Periwinkle] (g5a) at (72:.75cm){};
  \node[circle, fill=black, radius=3pt, inner sep=0pt, minimum size=4pt,fill=JungleGreen] (g4a) at (144:.75cm){};
  \node[circle, fill=black, radius=3pt, inner sep=0pt, minimum size=4pt,fill=blue] (g1a) at (216:.75cm){};
  \node[circle, fill=black, radius=3pt, inner sep=0pt, minimum size=4pt,fill=Dandelion] (g2a) at (288:.75cm){};
  \node (A) at (0,1){\small $A_5(R_W)$};
\end{scope}
\begin{scope}[shift={(1.5,1.5)}]
  \node[circle, fill=black, radius=3pt, inner sep=0pt, minimum size=4pt,fill=Orchid] (g1b) at (0:.75cm){};
  \node[circle, fill=black, radius=3pt, inner sep=0pt, minimum size=4pt,fill=Maroon] (g2b) at (90:.75cm){};
  \node[circle, fill=black, radius=3pt, inner sep=0pt, minimum size=4pt,fill=Cerulean] (g3b) at (180:.75cm){};
  \node[circle, fill=black, radius=3pt, inner sep=0pt, minimum size=4pt,fill=NavyBlue] (g4b) at (270:.75cm){};
  \node (A) at (0,1){\small $A_4(R_X)$};
\end{scope}
\begin{scope}[shift={(1.5,-1.5)}]
  \node[circle, fill=black, radius=3pt, inner sep=0pt, minimum size=4pt,fill=Lavender] (g3c) at (0:.75cm){};
  \node[circle, fill=black, radius=3pt, inner sep=0pt, minimum size=4pt,fill=CadetBlue] (g5c) at (72:.75cm){};
  \node[circle, fill=black, radius=3pt, inner sep=0pt, minimum size=4pt,fill=BrickRed] (g4c) at (144:.75cm){};
  \node[circle, fill=black, radius=3pt, inner sep=0pt, minimum size=4pt,fill=Blue] (g1c) at (216:.75cm){};
  \node[circle, fill=black, radius=3pt, inner sep=0pt, minimum size=4pt,fill=Goldenrod] (g2c) at (288:.75cm){};
  \node (A) at (0,-1){\small $A_5(R_Y)$};
\end{scope}
\begin{scope}[shift={(-1.5,-1.5)}]
  \node[circle, fill=black, radius=3pt, inner sep=0pt, minimum size=4pt,fill=Yellow] (g1d) at (0:.75cm){};
  \node[circle, fill=black, radius=3pt, inner sep=0pt, minimum size=4pt,fill=Tan] (g2d) at (120:.75cm){};
  \node[circle, fill=black, radius=3pt, inner sep=0pt, minimum size=4pt,fill=YellowGreen] (g3d) at (240:.75cm){};
  \node (A) at (0,-1){\small $A_3(R_Z)$};
\end{scope}

\begin{pgfonlayer}{bg}
\def\br{30}
\draw[->-,Apricot!\br] (g1a) to (g1b);
\draw[->-,Apricot!\br] (g1a) to (g2b);
\draw[->-,Apricot!\br] (g1a) to (g3b);
\draw[->-,Apricot!\br] (g1a) to (g4b);
\draw[->-,Apricot!\br] (g2a) to (g1b);
\draw[->-,Apricot!\br] (g2a) to (g2b);
\draw[->-,Apricot!\br] (g2a) to (g3b);
\draw[->-,Apricot!\br] (g2a) to (g4b);
\draw[->-,Apricot!\br] (g3a) to (g1b);
\draw[->-,Apricot!\br] (g3a) to (g2b);
\draw[->-,Apricot!\br] (g3a) to (g3b);
\draw[->-,Apricot!\br] (g3a) to (g4b);
\draw[->-,Apricot!\br] (g4a) to (g1b);
\draw[->-,Apricot!\br] (g4a) to (g2b);
\draw[->-,Apricot!\br] (g4a) to (g3b);
\draw[->-,Apricot!\br] (g4a) to (g4b);
\draw[->-,Apricot!\br] (g5a) to (g1b);
\draw[->-,Apricot!\br] (g5a) to (g2b);
\draw[->-,Apricot!\br] (g5a) to (g2b);
\draw[->-,Apricot!\br] (g5a) to (g3b);
\draw[->-,Apricot!\br] (g1b) to (g1c);
\draw[->-,Apricot!\br] (g1b) to (g2c);
\draw[->-,Apricot!\br] (g1b) to (g3c);
\draw[->-,Apricot!\br] (g1b) to (g4c);
\draw[->-,Apricot!\br] (g1b) to (g5c);
\draw[->-,Apricot!\br] (g2b) to (g1c);
\draw[->-,Apricot!\br] (g2b) to (g2c);
\draw[->-,Apricot!\br] (g2b) to (g3c);
\draw[->-,Apricot!\br] (g2b) to (g4c);
\draw[->-,Apricot!\br] (g2b) to (g5c);
\draw[->-,Apricot!\br] (g3b) to (g1c);
\draw[->-,Apricot!\br] (g3b) to (g2c);
\draw[->-,Apricot!\br] (g3b) to (g3c);
\draw[->-,Apricot!\br] (g3b) to (g4c);
\draw[->-,Apricot!\br] (g3b) to (g5c);
\draw[->-,Apricot!\br] (g4b) to (g1c);
\draw[->-,Apricot!\br] (g4b) to (g2c);
\draw[->-,Apricot!\br] (g4b) to (g3c);
\draw[->-,Apricot!\br] (g4b) to (g4c);
\draw[->-,Apricot!\br] (g4b) to (g5c);
\draw[->-,Apricot!\br] (g1c) to (g1d);
\draw[->-,Apricot!\br] (g1c) to (g2d);
\draw[->-,Apricot!\br] (g1c) to (g3d);
\draw[->-,Apricot!\br] (g2c) to (g1d);
\draw[->-,Apricot!\br] (g2c) to (g2d);
\draw[->-,Apricot!\br] (g2c) to (g3d);
\draw[->-,Apricot!\br] (g3c) to (g1d);
\draw[->-,Apricot!\br] (g3c) to (g2d);
\draw[->-,Apricot!\br] (g3c) to (g3d);
\draw[->-,Apricot!\br] (g4c) to (g1d);
\draw[->-,Apricot!\br] (g4c) to (g2d);
\draw[->-,Apricot!\br] (g4c) to (g3d);
\draw[->-,Apricot!\br] (g5c) to (g1d);
\draw[->-,Apricot!\br] (g5c) to (g2d);
\draw[->-,Apricot!\br] (g5c) to (g3d);
\draw[->-,Apricot!\br] (g1d) to (g1a);
\draw[->-,Apricot!\br] (g1d) to (g2a);
\draw[->-,Apricot!\br] (g1d) to (g3a);
\draw[->-,Apricot!\br] (g1d) to (g4a);
\draw[->-,Apricot!\br] (g1d) to (g5a);
\draw[->-,Apricot!\br] (g2d) to (g1a);
\draw[->-,Apricot!\br] (g2d) to (g2a);
\draw[->-,Apricot!\br] (g2d) to (g3a);
\draw[->-,Apricot!\br] (g2d) to (g4a);
\draw[->-,Apricot!\br] (g2d) to (g5a);
\draw[->-,Apricot!\br] (g3d) to (g1a);
\draw[->-,Apricot!\br] (g3d) to (g2a);
\draw[->-,Apricot!\br] (g3d) to (g3a);
\draw[->-,Apricot!\br] (g3d) to (g4a);
\draw[->-,Apricot!\br] (g3d) to (g5a);
\end{pgfonlayer}

  \draw[thick,-<>-] (g1a) to (g2a);
  \draw[thick,-<>-] (g2a) to (g3a);
  \draw[thick,-<>-] (g3a) to (g4a);
  \draw[thick,-<>-] (g1a) to (g3a);
  \draw[thick,-<>-] (g3a) to (g5a);
  \draw[thick,->-,gray] (g1a) to (g4a);
  \draw[thick,->-,gray] (g1a) to (g5a);
  \draw[thick,->-,gray] (g2a) to (g4a);
  \draw[thick,->-,gray] (g4a) to (g5a);
  
  \draw[thick,-<>-] (g2b) to (g3b);
  \draw[thick,-<>-] (g2b) to (g4b);
  \draw[thick,-<>-] (g3b) to (g4b);
  \draw[thick,->-,gray] (g1b) to (g2b);
  \draw[thick,->--,gray] (g1b) to (g3b);
  \draw[thick,->-,gray] (g1b) to (g4b);
  
  \draw[thick,-<>-] (g1c) to (g2c);
  \draw[thick,-<>-] (g1c) to (g3c);
  \draw[thick,-<>-] (g1c) to (g4c);
  \draw[thick,-<>-] (g1c) to (g5c);
  \draw[thick,-<>-] (g2c) to (g3c);
  \draw[thick,-<>-] (g4c) to (g5c);
  \draw[thick,->-,gray] (g3c) to (g4c);
  \draw[thick,->-,gray] (g3c) to (g5c);
  \draw[thick,->-,gray] (g2c) to (g4c);
  \draw[thick,->-,gray] (g2c) to (g5c);
  
  \draw[thick,-<>-] (g1d) to (g2d);
  \draw[thick,-<>-] (g2d) to (g3d);
  \draw[thick,-<>-] (g1d) to (g3d);
 \end{tikzpicture}
 \caption{An example of how to construct a directed graph with a specified number of directed cliques within each module. \textbf{A}: Four modules connected in a cycle, with the specified number of vertices ($n$) and the required number of directed clique ($C_n$). \textbf{B}: Four posets whose number of linear extensions equals the required number of directed cliques for the corresponding module in \textbf{A}. \textbf{C}: A directed graph with the required number of directed cliques in each module, obtained by replacing each module with the directed graph corresponding to the poset in \textbf{B}.}\label{fig:construct_graph}
 \end{figure}
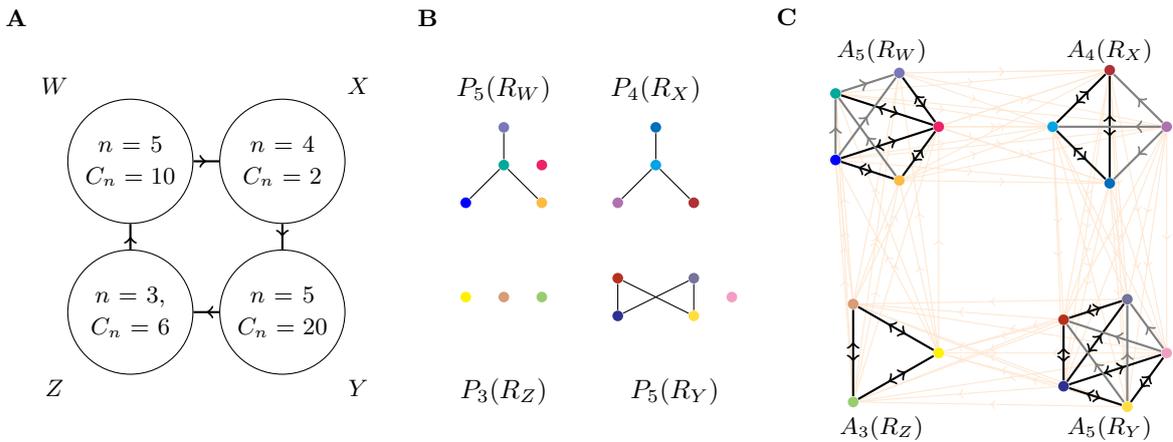

\section*{Acknowledgements}
The authors thank Kathryn Hess and J\=anis Lazovskis for insightful discussions.

\bibliographystyle{alpha}
\bibliography{bibfile}

\end{document}